\documentclass[8pt]{article}
\usepackage{amssymb}
\usepackage{setspace}
\usepackage[british]{babel}
\usepackage[mathscr]{euscript}
\usepackage{dsfont}
\usepackage[bottom=3cm, top = 3cm, left = 2cm, right = 2cm]{geometry}
\usepackage{authblk}
\usepackage{amsmath,amsfonts,amssymb,amsthm,booktabs,color,epsfig,graphicx,hyperref,url, xcolor, arydshln, parskip, enumitem, subcaption}

\usepackage[titletoc]{appendix}

\hypersetup{
    colorlinks=true,
    linkcolor=blue,
    urlcolor=blue,
    linktoc=all,
    citecolor=blue
           }

\newtheorem{definition}{Definition}

\newtheorem{lemma}{Lemma}
\newtheorem{theorem}{Theorem}
\newtheorem{proposition}{Proposition}
\newtheorem{corollary}{Corollary}
\newtheorem{remark}{Remark}

\newtheorem{innerassumption}{Assumption}
\newenvironment{Assumption}[1]
  {\renewcommand\theinnerassumption{#1}\innerassumption}
  {\endinnerassumption}
  
\makeatletter
\renewenvironment{proof}[1][\proofname]{\par
\pushQED{\qed}%
\normalfont \topsep6\p@\@plus6\p@\relax
\trivlist
\item\relax
{\textbf{
#1\@addpunct{ }}}\hspace\labelsep\ignorespaces
}{%
\popQED\endtrivlist\@endpefalse
}
\makeatother

\title{Non-parametric estimator of a multivariate madogram for missing-data and extreme value framework}

\author[1]{Alexis Boulin\footnote{Corresponding author. Email address: \url{Alexis.Boulin@unice.fr}}}
\author[1]{Elena Di Bernardino}
\author[1]{Thomas Lalo\"{e}}
\author[2]{Gwladys Toulemonde}
    
\affil[1]{Laboratoire J.A. Dieudonn\'{e}, UMR CNRS 7351, Universit\'{e} C\^{o}te d'Azur,  Nice, France}
\affil[2]{Univ Montpellier, CNRS, Inria, Montpellier, France}

\date{}

\begin{document}

    \maketitle

    \begin{abstract}
        The modeling of dependence between maxima is an important subject in several applications in risk analysis. To this aim, the extreme value copula function, characterised via the madogram, can be used as a margin-free description of the dependence structure. From a practical point of view, the family of extreme value distributions is very rich and arise naturally as the limiting distribution of properly normalised component-wise maxima. In this paper, we investigate the nonparametric estimation of the madogram where data are completely missing at random. We provide the functional central limit theorem for the considered multivariate madrogram correctly normalized, towards a tight Gaussian process for which the covariance function depends on the probabilities of missing. Explicit formula for the asymptotic variance is also given. Our results are illustrated in a finite sample setting  with a simulation study.
        
        
        \smallskip
        \noindent \href{https://mathscinet.ams.org/mathscinet/msc/msc2020.html}{MSC2020 subject classifications} : 62D10, 62G05, 62G20, 62G32, 62H10, 62H12.
    \end{abstract}

    \section{Introduction\label{sec:1}}

    Management of environmental ressources often requires the analysis of multivariate extreme values. In climate studies, extreme events represent a major challenge due to their consequences. The problem of missing data is present in many fields in particular in environmental research (see \cite{XIA1999131}, or Section 2 in \cite{saunders}), usually due to instruments, communication and processing errors. In a time series setting, the observation periods of a multivariate series could be different and overlap only partially. The problem of estimating when unequal amounts of data are available to each variable is meaningful in many applications for financial economics where data cannot be generated as neatly overlapping samples (see \cite{Patton06estimationof}). Missing values in dependence modeling is of a prime interest as the nonparametric estimation of the empirical copula process has been tackled by \cite{segers2014hybrid} under the \textbf{M}issing \textbf{C}ompletely \textbf{A}t \textbf{R}andom (\textbf{MCAR}) condition.  In this paper, we consider nonparametric methods for assessing extremal dependencies involving variables with missing values under \textbf{MCAR} condition. We are particularly interested in the dependence structure of multivariate extreme value distribution. Formally, this concept is defined as follows.\smallskip
        
    Let $(\Omega, \mathcal{A}, \mathbb{P})$ be a probability space and $\textbf{X} = (X_1, \dots, X_d)$ be a $d$-dimensional random vector with values in $(\mathbb{R}^d, \mathcal{B}(\mathbb{R}^d))$, with $d \geq 2$. This random vector has a joint distribution function $F$ and its margins are denoted by $F_j(x) = \mathbb{P} \{X_j \leq x\}$ for all $x \in \mathbb{R}$ and $j \in \{1,\dots,d\}$. A function $C : [0,1]^d \rightarrow [0,1]$ is called a $d$-dimensional copula if it is the restriction to $[0,1]^d$ of a distribution function whose margins are given by the uniform distribution on the interval $[0,1]$. Since the work of \cite{Skla59}, it is well known that every distribution function $F$ can be decomposed as $F(\textbf{x}) = C(F_1(x_1), \dots, F_d(x_d))$, for all $\textbf{x} \in \mathbb{R}^d$ and the copula $C$ is unique if the marginals are continuous. We will  consider in the rest of the paper a $d$-dimensional random vector $\textbf{X}$ which distribution is a multivariate extreme value distribution $F$,  \emph{i.e.}, its one dimensional distributions are Generalized Extreme-Value (GEV) distributions and the copula $C$ is an extreme value copula (see \cite{gudendorf2009extremevalue}),  defined by
    \begin{equation}
    \label{evc}
        C(\textbf{u}) = \exp\left( -\ell(-\ln(u_1), \dots, -\ln(u_d)) \right), \quad \textbf{u} \in (0,1]^d,
    \end{equation}
    with $\ell : [0,\infty)^d \rightarrow [0,\infty)$ the stable tail dependence function which is convex, homogeneous of order one, namely $\ell(c x_1, \dots, c x_d) = c \ell(x_1,\dots,x_d)$ for $c > 0$ and satisfies $\max(x_1,\dots, x_d) \leq \ell(x_1,\dots,x_d) \leq x_1 + \dots + x_d$, $\forall (x_1,\dots,x_d) \in [0,\infty)^d$. Denote by $\Delta^{d-1} = \{(w_1, \dots, w_d) \in [0,1]^d : w_1 + \dots + w_d = 1\}$ the unit simplex. By homogeneity, $\ell$ is characterized by the \emph{Pickands dependence function} $A: \Delta^{d-1} \rightarrow [1/d,1]$, which is the restriction of $\ell$ to the unit simplex $\Delta^{d-1}$ :
    \begin{equation}
        \label{pick_tail}
        \ell(x_1, \dots, x_d) = (x_1 + \dots + x_d) A(w_1, \dots, w_d), \quad w_j = \frac{x_j}{x_1 + \dots + x_d},
    \end{equation}
    for $j\in \{2,\dots,d\}$ and $w_1 = 1-w_2-\dots-w_d$ with $(x_1,\dots,x_d) \in [0,\infty)^d \setminus \{\textbf{0}\}$. Notice that, for every $\textbf{w} \in \Delta^{d-1}$ and $u \in ]0,1[$
    \begin{equation}
        \label{evc_section}
        C(u^{w_1}, \dots, u^{w_d}) = u^{A(\textbf{w})}.
    \end{equation}
    Based on the madogram concept from geostatistics, the $\lambda$-madogram is introduced in \cite{naveau:hal-00312758} to capture bivariate extremal dependencies. The generalization of the $\lambda$-madogram was previously proposed by \cite{fonseca2015generalized} and \cite{MARCON20171}, this quantity is defined in the latter as:
    \begin{equation}
        \label{w-mado}
        \nu(\textbf{w}) = \mathbb{E}\left[\bigvee_{j=1}^d \left\{ F_j(X_j) \right\}^{1/w_j} - \frac{1}{d} \sum_{j=1}^d \left\{F_j(X_j)\right\}^{1/w_j} \right],
    \end{equation}
    if $w_j = 0$ and $0<u<1$, then $u^{1/w_j} = 0$ by convention. The $\textbf{w}$-madogram can be interpreted as the $L_1$-distance between the maximum and the average of the uniform margins $F_1(X_1), \dots, F_d(X_d)$ elevated to the inverse of the corresponding weights $w_1, \dots, w_d$. This quantity describes the dependence structure between extremes by its relation with the Pickands dependence function as stated by the Proposition 2.2 of  \cite{MARCON20171}, namely
    \begin{equation}
        \label{w_mado_pickands}
        A(\textbf{w}) = \frac{\nu(\textbf{w})+c(\textbf{w})}{1-\nu(\textbf{w})-c(\textbf{w})},
    \end{equation}
    with $c(\textbf{w}) = d^{-1} \sum_{j=1}^d w_j/(1+w_j)$. Through this relation, it contributes to the vast literature of the estimation of the Pickands dependence function for bivariate extreme value copula (see \cite{10022049959}, \cite{DEHEUVELS1991429}, \cite{Capra1997ANE}, \cite{1081282691}) and extended to the multivariate extreme value copula (see for example \cite{GUDENDORF20123073}). Also, a test for assessing asymptotic independence in dimension $d\geq 2$ has been designed based on the $\textbf{w}$-madogram (see \cite{GUILLOU2018114}). Several methods for handling missing values in the framework of extremes have been proposed for univariate time series (see \emph{e.g.} \cite{AndreiaScotto, Ferreira2021}). However, handling missing values in the context of multivariate extreme values with $d\geq2$ is still in their infancy.
    
    \paragraph{Main results}
    
    The main contribution of this paper is to give an estimator of the $\textbf{w}$-madogram in \eqref{w-mado} involving variables with missing values and to study its asymptotic properties. As far as we know, only \cite{GuiNa} detailed the variance for the madogram of a bivariate random vector while taking the independent copula and found $1/90$. In this paper we propose improvements in three directions : we consider a general multidimensional case ($d\geq 2$), we deal with missing data and we consider a dependence structure given by an extreme value copula. Thus, we present in Theorem \ref{weak_conv_hybrid_mado} a functional central limit theorem that gives the weak convergence for the considered multivariate madogram towards a tight Gaussian process for which the covariance function depends on the probabilities of missing. When the trajectory of our empirical process is fixed, we show in Proposition \ref{Boulin} the asymptotic normality of the estimator of the multivariate madogram where explicit formula for the asymptotic variance is also given. These results are transposed to the estimation of the Pickands dependence function with missing data in Corollary \ref{asymptotic_distribution_pickands} by the use of the functional delta method.
            
    \paragraph{Notations}         
    The symbol $\triangleq$ means to be equal to. In order to shorten formulas, notations
    \begin{align*}
        &\textbf{u}_j(t) \triangleq (u_1, \dots, u_{j-1}, t, u_{j+1}, \dots, u_d), \\
        &\textbf{u}_{jk}(s,t) \triangleq (u_1,\dots,u_{j-1},s,u_{j+1}, \dots, u_{k-1}, t, u_{k+1}, \dots,u_d),
    \end{align*}
    will be adopted for $s,t \in [0,1]$, $(u_1,\dots,u_{j-1},u_{j+1}, \dots, u_d) \in [0,1]^{d-1}$ and $j,k \in \{1,\dots,d\}$ with $j<k$. The notation $\textbf{1}$ (resp. $\textbf{0}$) corresponds to the $d$-dimensional vector composed out of $1$ (resp. $0$). Similarly, we define $\textbf{1}_j(s)$, $\textbf{0}_j(s)$, $\textbf{1}_{jk}(s,t)$ and $\textbf{0}_{jk}(s,t)$ with the same idea of previous notations of this paragraph.\smallskip
            
    The following notations are also used. Given $\mathcal{X}$ an arbitrary set, let $\ell^\infty(\mathcal{X})$ denote the space of bounded real-valued functions on $\mathcal{X}$. For $f: \mathcal{X} \rightarrow \mathbb{R}$, let $||f||_\infty = sup_{x \in \mathcal{X}} |f(x)|$. Here, we use the abbreviation $Q(f) = \int f dQ$ for a given measurable function $f$ and signed measure $Q$. The arrows $\overset{a.s.}{\rightarrow}$, $\overset{d}{\rightarrow}$ denote almost sure convergence and convergence in distribution of random vectors. Weak convergence of a sequence of maps will be understood in the sense of J.Hoffman-J{\o}rgensen (see Part 1 in  \cite{vaartwellner96book}). Given that $n \in \mathbb{N}^*, X, X_n$ are maps from $(\Omega, \mathcal{A}, \mathbb{P})$ into a metric space $\mathcal{X}$ and that $X$ is Borel measurable, $(X_n)_{n\geq 1}$ is said to converge weakly to $X$ if $\mathbb{E}^* f(X_n) \rightarrow \mathbb{E}f(X)$ for every bounded continuous real-valued function $f$ defined on $\mathcal{X}$, where $\mathbb{E}^*$ denotes outer expectation in the event that $X_n$ may not be Borel measurable. In what follows, weak convergence is denoted by $X_n \rightsquigarrow X$.\smallskip
            
    The paper is organised as follows: We propose in Section \ref{theory} estimators of the $\textbf{w}$-madogram suitable to the missing data framework. We state the weak convergence of the depicted estimators. Explicit formula for the asymptotic variance are also given. In Section \ref{numerical_results}, we illustrate the performance of the considered estimator in the finite-sample framework. Section \ref{sec:4} is devoted to apply our method on a dataset with missing data and non-concomittant record periods of annual maxima rainfall in Central Eastern Canada. A discussion on our assumptions and possible extensions of this  work are presented in Section \ref{sec:4}. All the proofs are postponed to the   \ref{proof}.

    \section{Non parametric estimation of the Madogram with missing data}
            \label{theory}
            
            We consider independent and identically distributed ($\emph{i.i.d.}$) copies $\textbf{X}_1, \dots, \textbf{X}_n$ of $\textbf{X}$. In presence of missing data, we do not observe a complete vector $\textbf{X}_i$ for $i \in \{1,\dots,n\}$. We introduce $\textbf{I}_i \in \{0,1\}^d$ which satisfies, $\forall j \in \{1,\dots,d\}$, $I_{i,j} = 0$ if $X_{i,j}$ is not observed. To formalize incomplete observations, we introduce the incomplete vector $\tilde{\textbf{X}}_i$ with values in the product space $\bigotimes_{j=1}^d (\mathbb{R} \cup \{\textup{NA}\})$ (where NA denotes a missing data) such as
            \begin{equation*}
                \tilde{X}_{i,j} = X_{i,j} I_{i,j} + \textup{NA} (1-I_{i,j}), \quad i \in \{1,\dots,n\}, \, j \in \{1,\dots, d\}.
            \end{equation*}
            We thus suppose that we observe a $2d$-tuple such as
            \begin{equation}
                \label{missing_2}
                (\textbf{I}_i, \tilde{\textbf{X}}_i), \quad i \in \{1,\dots,n\},
            \end{equation}
            \emph{i.e.} at each $i \in \{1,\dots,n\}$, several entries may be missing. We also suppose that for all $i \in \{1, \dots,n \}$, $\textbf{I}_{i}$ are \emph{i.i.d} copies from $\textbf{I} = (I_1,\dots, I_d)$ where $I_j$ is distributed according to a Bernoulli random variable $\mathcal{B}(p_j)$ with $p_j = \mathbb{P}(I_j = 1)$ for $j \in \{1,\dots,d\}$. We denote by $p$ the probability of observing completely a realization from $\textbf{X}$, that is $p = \mathbb{P}(I_1=1, \dots, I_d = 1)$. Let us now define the empirical cumulative distribution in case of missing data, we write for notational convenience $\{ \tilde{\textbf{X}}_i \leq \textbf{x}\} \triangleq \{\tilde{X}_{i,1} \leq x_1, \dots, \tilde{X}_{i,d} \leq x_d\}$ and $n_j = \sum_{i=1}^n I_{i,j}$,
            \begin{align}
                \label{F_X_G_Y_missing}
                &\hat{F}_{n,j} (x) = \frac{\sum_{i=1}^n \mathds{1}_{\{\tilde{X}_{i,j} \leq x\}}I_{i,j}}{n_j}, \, \forall x \in \mathbb{R}, \quad \hat{F}_n(\textbf{x}) = \frac{\sum_{i = 1}^n \mathds{1}_{\{\tilde{\textbf{X}}_i \leq \textbf{x}\}} \Pi_{j=1}^d I_{i,j}}{\sum_{i = 1}^n  \Pi_{j=1}^dI_{i,j}}, \, \forall \textbf{x} \in \mathbb{R}^d,
            \end{align}
            where $\{\tilde{X}_{i,j} \leq x\} = \emptyset$ (\emph{resp.} $\{\tilde{\textbf{X}}_{i} \leq \textbf{x}\} = \emptyset$) if $\tilde{X}_{i,j} = \textup{NA}$ (\emph{resp.} if there exists $j \in \{1, \dots, d\}$ such that $\tilde{X}_{i,j} = \textup{NA}$). The idea raised here is to estimate non parametrically the margins using all available data of the corresponding series. To avoid dealing with points at the boundary of the unit square, it is more convenient to work with scaled ranks (see for example \cite{10.1214/08-AOS672}) defined explicitely by
            \begin{equation}
                \label{eq:rank_based}
                 \widetilde{U}_{i,j} =  \frac{n_j}{n_j+1} \hat{F}_{n,j} (\tilde{X}_{i,j})=\frac{1}{n_j+1} \sum_{k=1}^n \mathds{1}_{\{ \tilde{X}_{k,j} \leq \tilde{X}_{i,j}\}} I_{i,j}, \quad j \in \{1,\dots,d\}.
            \end{equation}
            We recall the definition of the \emph{hybrid copula estimator} introduced by \cite{segers2014hybrid}
            \begin{equation*}
                \hat{C}_n^\mathcal{H}(\textbf{u}) = \hat{F}_n(\hat{F}_{n,1}^\leftarrow(u_1), \dots, \hat{F}_{n,d}^\leftarrow(u_d)), \quad \textbf{u} \in [0,1]^d,
            \end{equation*}
            where $\hat{F}_{n,j}^{\leftarrow}$ denotes the generalized inverse function of $\hat{F}_{n,j}$ for $j \in \{1,\dots,d\}$, \emph{i.e.} $\hat{F}_{n,j}^\leftarrow(u) = \inf \{ x\in \mathbb{R} | \hat{F}_{n,j}(x) \geq u \}$ with $0<u<1$. The normalized estimation error of the hybrid copula estimator is
            \begin{equation}
                \label{hybrid_copula}
                \mathbb{C}_n^\mathcal{H}(\textbf{u}) = \sqrt{n} \left( \hat{C}_{n}^\mathcal{H}(\textbf{u}) - C(\textbf{u}) \right), \quad \textbf{u} \in [0,1]^d.
            \end{equation}
            On the condition that the first-order partial derivatives of the copula function $C$ exists and are continuous on a subset of the unit hypercube, \cite{10.3150/11-BEJ387} obtained weak convergence of the normalized estimation error of the classical empirical copula process (see \cite{1132813}). To satisfy this condition, we introduce the following assumption as suggested in \cite{10.3150/11-BEJ387} (see Example 5.3).
            \begin{Assumption}{A}
                \label{Cond_1}
                \begin{enumerate}
                    \item[]
                    \item The distribution function $F$ has continuous margins $F_1, \dots, F_d$. \label{Cond_1_i}
                    \item For every $j \in \{1,\dots,d\}$, the first-order partial derivative $\dot{\ell}_j$ of $\ell$ with respect to $x_j$ exists and is continuous on the set $\{x \in [0,\infty)^d: x_j > 0\}$. \label{Cond_1_ii}
                \end{enumerate} 
            \end{Assumption}
            The Assumption \ref{Cond_1}\ref{Cond_1_i} guarantees that the representation $F(\textbf{x}) = C(F_1(x_1),\dots, F_d(x_d))$ is unique on the range of $(F_1, \dots, F_d)$. Under the Assumption \ref{Cond_1}\ref{Cond_1_ii}, the first-order partial derivatives of $C$ with respect to $u_j$ denoted as $\dot{C}_j$ exists and are continuous on the set $\{ \textbf{u} \in [0,1]^d : 0 < u_j < 1 \}$. We now propose an estimator of the $\textbf{w}$-madogram defined in Equation (\ref{w-mado}) under a general context with possible missing data.
            \begin{definition}
            \label{def_hybrid_lambda_fmado}
                Let $(\textbf{I}_i, \tilde{\textbf{X}}_i)_{i=1}^n$ be a sample given by Equation \eqref{missing_2}, we define the hybrid nonparametric estimator of the $\textbf{w}$-madogram in Equation \eqref{w-mado} by
                \begin{equation}
                    \label{hybrid_lambda_fmado}
                    \hat{\nu}_n^\mathcal{H}(\textbf{w}) = \frac{1}{\sum_{i=1}^n \Pi_{j=1}^d I_{i,j}} \sum_{i = 1}^n \left[ \left(\bigvee_{j=1}^d \widetilde{U}_{i,j}^{1/w_j} - \frac{1}{d} \sum_{j=1}^d\widetilde{U}_{i,j}^{1/w_j}\right)\Pi_{j=1}^d I_{i,j} \right],
                \end{equation}
                where $\widetilde{U}_{i,j}$  are scaled ranks defined as in  Equation \eqref{eq:rank_based}.
            \end{definition}
            The intuitive idea here is to estimate the margins using all available data from the corresponding variables and estimate $\nu(\textbf{w})$ using only the overlapping data. Notice  that in the complete data framework, \emph{i.e.} when $p = 1$ we retrieve a variation of the $\textbf{w}$-madogram such as defined in \cite{MARCON20171}, namely
            \begin{equation*}
                \hat{\nu}_n(\textbf{w}) = \frac{1}{n} \sum_{i = 1}^n \left[ \bigvee_{j=1}^d\widetilde{U}_{i,j}^{1/w_j} - \frac{1}{d} \sum_{j=1}^d\widetilde{U}_{i,j}^{1/w_j}\right],
            \end{equation*}
            with $\widetilde{U}_{i,j}$  in $\{1/(n+1), \dots, n/(n+1)\}$.\smallskip
            
            Note that the theoretical quantity defined in \eqref{w-mado} does verify endpoint constraints, \emph{i.e.} $\nu(\textbf{e}_j) = (d-1)/2d$ for all $j\in \{1,\dots,d\}$ where $\textbf{e}_j$ is the jth vector of the canonical basis. 
            \begin{remark}
                \label{remark_corr}
                Unlike $\nu$, the estimator defined in \eqref{hybrid_lambda_fmado} does not verify the endpoints constraints. In addition, the variance at $\textbf{e}_j$ does not equal 0. Indeed, suppose that we evaluate this statistic at $\textbf{w} = \textbf{e}_j$ as $\widetilde{U}_{i,j} \in (0,1)$ for every $i \in \{1,\dots,n\}$ and $j \in \{1,\dots, d\}$ we obtain the following estimator
                \begin{equation*}
                    \hat{\nu}_n^\mathcal{H}(\textbf{e}_j) = \frac{1}{\sum_{i=1}^n \Pi_{j=1}^d I_{i,j}} \sum_{i=1}^n\left[ \widetilde{U}_{i,j} - \frac{1}{d}\widetilde{U}_{i,j} \right]\Pi_{j=1}^d I_{i,j}.
                \end{equation*}
                In this situation, the sample $\left(\widetilde{U}_{i,1}, \dots, \widetilde{U}_{i,j-1}, \widetilde{U}_{i,j+1}, \dots, \widetilde{U}_{i,d}\right)_{i=1}^n$ is taken into account through the indicators sequence $(I_{i,1}, \dots, I_{i,j-1}, I_{i,j+1}, \dots, I_{i,d})_{i=1}^n$ and induces a supplementary variance when estimating. 
            \end{remark}
            Proceeding as in \cite{naveau:hal-00312758} for the bivariate case and complete data framework, we propose below a  modified estimator which satisfies the endpoint constraints in the general multivariate framework with possible missing data. 
            \begin{definition}
                Let $(\textbf{I}_i, \tilde{\textbf{X}}_i)_{i=1}^n$ be a sample given by Equation \eqref{missing_2} and $\hat{\nu}_n^\mathcal{H}(\textbf{w})$ be as in \eqref{hybrid_lambda_fmado}. Given continuous functions $\lambda_1, \dots, \lambda_d : \Delta^{d-1} \rightarrow \mathbb{R}$ verifying $\lambda_j(\textbf{e}_k) = \delta_{jk}$ (the Kronecker delta) for $j,k \in \{1,\dots,d\}$, we define the hybrid corrected estimator of the $\textbf{w}$-madogram by
                \begin{align}
                \hat{\nu}_n^{\mathcal{H}*}(\textbf{w}) = \hat{\nu}_n^\mathcal{H}(\textbf{w})
                -\sum_{j=1}^d \frac{\lambda_j(\textbf{w})(d-1)}{d} \left[\frac{1}{\sum_{i=1}^n \Pi_{j=1}^d I_{i,j}} \sum_{i=1}^n \left( \widetilde{U}_{i,j}^{1/w_j}\Pi_{j=1}^d I_{i,j} \right) -\frac{w_j}{1+w_j} \right].
                \label{corrected_lambda_FMado_hybrid}
            \end{align}
            \end{definition}
            \begin{remark}
                \label{rem:remark_2}
                One has often that endpoint corrections do not have an impact to the asymptotic behavior with complete data framework and unknown margins (see Section 2.3 and 2.4 of \cite{10.1214/08-AOS672}). That is not always the case in the missing data framework and this feature is of interest as discussed in Remark \ref{remark_corr}. 
            \end{remark}
            In the following we prove a functional central limit theorem (see Theorem \ref{weak_conv_hybrid_mado}) concerning the weak convergence of the following processes
            \begin{equation}
                \label{processes}
                \sqrt{n} \left(\hat{\nu}_n^{\mathcal{H}}(\textbf{w}) - \nu(\textbf{w}) \right)_{\textbf{w} \in \Delta^{d-1}}, \quad  \sqrt{n} \left(\hat{\nu}_n^\mathcal{H*}(\textbf{w}) - \nu(\textbf{w}) \right)_{\textbf{w} \in \Delta^{d-1}}.
            \end{equation}
            Before presenting this result, we introduce below a specific assumption on the missing mechanism.
            \begin{Assumption}{B}
                \label{Cond_2}
                We suppose that for all $i \in \{1, \dots, n\}$, the vector $\textbf{I}_i$ and $\textbf{X}_i$ are independent, i.e. the data are missing completely at random (\textbf{MCAR}).
            \end{Assumption}
            
            Without missing data, the weak convergence of the normalized estimation error of the empirical copula process has been proved by \cite{10.2307/3318798} under a more restrictive condition than Assumption \ref{Cond_1}. The difference being that $C$ should be continuously differentiable on the closed hypercube. Denoting by $D([0,1]^2)$ the Skorohod space, this statement makes use of previous results on the Hadamard differentiability of the map $\phi : D([0,1]^2) \rightarrow \ell^\infty([0,1]^2)$ which transforms the cumulative distribution function $F$ into its copula function $C$ (see also Lemma 3.9.28 from \cite{vaartwellner96book}). With the hybrid copula estimator, we need a technical assumption in order to guarantee the weak convergence of the process $\mathbb{C}_n^\mathcal{H}$ in \eqref{hybrid_copula} (see \cite{segers2014hybrid}). We note for convenience marginal distributions and quantile functions into vector valued functions $\textbf{F}_d$ and $\textbf{F}_d^\leftarrow$:
            \begin{align*}
                &\textbf{F}_d(\textbf{x}) = (F_1(x_1), \dots, F_d(x_d)), \quad \textbf{x} \in \mathbb{R}^d, 
                &\textbf{F}_d^\leftarrow(\textbf{u}) = (F_1^\leftarrow(u_1), \dots, F_d^\leftarrow(u_d)), \quad \textbf{u} \in [0,1]^d.
            \end{align*}
            \begin{Assumption}{C}
                \label{Cond_3}
                In the space $\ell^{\infty} (\mathbb{R}^d) \otimes (\ell^{\infty}(\mathbb{R}),\dots, \ell^{\infty}(\mathbb{R}))$ equipped with the topology of uniform convergence, we have the joint weak convergence 
                \begin{align*}
                    \left(\sqrt{n} (\hat{F}_n - F); \sqrt{n}(\hat{F}_{n,1} - F_1), \dots,\sqrt{n}(\hat{F}_{n,d} - F_d) \right) 
                    \rightsquigarrow \left(\alpha \circ \textbf{F}_d, \beta_1 \circ F_1, \dots,\beta_d \circ F_d\right),
                \end{align*}
                where the stochastic processes $\alpha$ and $\beta_j, j \in \{1,\dots,d\}$ take values in $\ell^{\infty} ([0,1]^d)$ and $\ell^{\infty} ([0,1])$ respectively, and are such that $\alpha \circ F$ and $\beta_j \circ F_j$ have continuous trajectories on $[-\infty, \infty]^d$ and $[-\infty, \infty]$ almost surely.
            \end{Assumption}
            Under Assumptions \ref{Cond_1} and \ref{Cond_3}, the stochastic process $\mathbb{C}_n^\mathcal{H}$ in \eqref{hybrid_copula} converges weakly to the tight Gaussian process $S_C$ defined by
            \begin{equation}
                \label{Segers_process}
                S_C(\textbf{u}) = \alpha(\textbf{u}) -\sum_{j=1}^d \dot{C}_j(\textbf{u})\beta_j(u_j), \quad \forall \textbf{u} \in [0,1]^d.
            \end{equation}
            Lemma \ref{lemma_1} in \ref{proof} states that the estimator $\hat{F}_n$ of the joint distribution and estimators of margins $\hat{F}_{n,j}$ defined in Equation \eqref{F_X_G_Y_missing}   verify Assumption \ref{Cond_3} (see \ref{proof} for details). We now have all tools in hand to consider the weak convergence of the stochastic processes in Equation \eqref{processes}. We note by $\{ \textbf{X} \leq \textbf{F}_d^{\leftarrow}(\textbf{u})\} = \{ X_1 \leq F_1^\leftarrow(u_1), \dots, X_d \leq F_d^\leftarrow(u_d)\}$.
            \begin{theorem}
                \label{weak_conv_hybrid_mado}
                Let $\mathbb{G}$ a tight Gaussian process and continuous functions $\lambda_1, \dots, \lambda_d : \Delta^{d-1} \rightarrow \mathbb{R}$ verifying $\lambda_j(\textbf{e}_k) = \delta_{jk}$. If $C$ is an extreme value copula with Pickands dependence function $A$ and under Assumptions \ref{Cond_1} and \ref{Cond_2}, we have the weak convergence in $\ell^\infty(\Delta^{d-1})$ for hybrid estimators defined in Equations (\ref{hybrid_lambda_fmado}) and (\ref{corrected_lambda_FMado_hybrid}), as $n \rightarrow \infty$,
                \begin{align*}
                    \sqrt{n} \left(\hat{\nu}_n^\mathcal{H}(\emph{\textbf{w}}) - \nu(\emph{\textbf{w}}) \right)_{\textbf{w} \in \Delta^{d-1}} &\rightsquigarrow \bigg( \frac{1}{d} \sum_{j=1}^d \int_{[0,1]} \alpha(\boldsymbol{1}_j(x^{w_j})) - \beta_j(x^{w_j})dx - \int_{[0,1]} S_C(x^{w_1}, \dots, x^{w_d}) dx \bigg)_{\emph{\textbf{w}} \in \Delta^{d-1}}, \\
                    \sqrt{n} \left(\hat{\nu}_n^{\mathcal{H}*}(\emph{\textbf{w}}) - \nu(\emph{\textbf{w}}) \right)_{\textbf{w} \in \Delta^{d-1}} &\rightsquigarrow \bigg( \frac{1}{d} \sum_{j=1}^d(1+\lambda_j(\emph{\textbf{w}})(d-1)) \int_{[0,1]} \alpha(\boldsymbol{1}_j(x^{w_j})) - \beta_j(x^{w_j})dx \\ & - \int_{[0,1]} S_C(x^{w_1}, \dots, x^{w_d}) dx \bigg)_{\emph{\textbf{w}} \in \Delta^{d-1}},
                \end{align*}
                where $S_C$ is defined in \eqref{Segers_process}, $\alpha(\emph{\textbf{u}}) = p^{-1} \mathbb{G}(\mathds{1}_{\{\mathbf{X} \leq \mathbf{F}_d^{\leftarrow}(\mathbf{u}),\textbf{I}=\textbf{1}\}} - C(\emph{\textbf{u}}) \mathds{1}_{\{\textbf{I}=\textbf{1}\}})$ and $\beta_j(u_j) = p_j^{-1} \mathbb{G}(\mathds{1}_{\{X_j \leq F_j^{\leftarrow}(u_j), I_j = 1\}} - u_j \mathds{1}_{\{I_j = 1\}})$ for $j \in \{1, \dots, d\}$ and $\emph{\textbf{u}} \in [0,1]^d$. For $(\emph{\textbf{u}},\emph{\textbf{v}},v_k) \in [0,1]^{2d+1}$, for $j \in \{1,\dots,d\}$ and $j < k$ the covariance functions of the processes $\alpha$ and $\beta_j$ are given by
                    \begin{align*}
                        & cov\left(\beta_j(u_j), \beta_j(v_j) \right) = p_j^{-1}\left( u_j \wedge v_j - u_j v_j \right), \\
                        & cov\left(\beta_{j}(u_j), \beta_k(v_k) \right) = \frac{p_{jk}}{p_j p_k} \left( C(\boldsymbol{1}_{j,k}(u_j,v_k)) - u_jv_k \right),
                    \end{align*}
                    and
                    \begin{align*}
                        & cov\left(\alpha(\emph{\textbf{u}}), \alpha(\emph{\textbf{v}})\right) =  p^{-1} \left( C(\emph{\textbf{u}} \wedge \emph{\textbf{v}}) - C(\emph{\textbf{u}}) C(\emph{\textbf{u}}) \right), \\
                        & cov\left(\alpha(\emph{\textbf{u}}), \beta_j(v_j)\right) = p_j^{-1}\left( C(\emph{\textbf{u}}_j(u_j\wedge v_j)) - C(\emph{\textbf{u}}) v_j \right),
                    \end{align*}
                    where $\emph{\textbf{u}} \wedge \emph{\textbf{v}}$ denotes the vector of componentwise minima and $p_{jk} = \mathbb{P}(I_j = 1, I_k = 1)$. 
            \end{theorem}
            We use empirical process arguments formulated in \cite{vaartwellner96book} to establish such a result. Details can be found in \ref{proof_weak_conv_hybrid_mado}. The following proposition states the asymptotic distribution of the estimators and gives explicit formula for the asymptotic variances for a fixed element of the unit simplex $\Delta^{d-1}$.
            \begin{proposition}
                \label{Boulin}
                Let $\emph{\textbf{p}} = (p_1, \dots, p_d, p)$ and $\emph{\textbf{w}} \in \Delta^{d-1}$, under the framework of Theorem \ref{weak_conv_hybrid_mado}, we have
                \begin{align*}
                    \sqrt{n} \left(\hat{\nu}_n^\mathcal{H}(\emph{\textbf{w}}) - \nu(\emph{\textbf{w}}) \right) \overunderset{d}{n \rightarrow \infty}{\rightarrow} \mathcal{N}\left(0, \mathcal{S}^{\mathcal{H}}(\emph{\textbf{p}}, \emph{\textbf{w}})\right), \quad \sqrt{n} \left(\hat{\nu}_n^{\mathcal{H}*}(\emph{\textbf{w}}) - \nu(\emph{\textbf{w}}) \right) \overunderset{d}{n \rightarrow \infty}{\rightarrow} \mathcal{N}\left(0, \mathcal{S}^{\mathcal{H}*}(\emph{\textbf{p}}, \emph{\textbf{w}})\right).
                \end{align*}
                Moreover the asymptotic variances   are given by
                \begin{align*}
                    \mathcal{S}^{\mathcal{H}}(\emph{\textbf{p}}, \emph{\textbf{w}}) =& \frac{1}{d^2} \sum_{j=1}^d (p^{-1} - p_j^{-1}) \sigma^2_j(\emph{\textbf{w}}) + \sigma_{d+1}^2(\emph{\textbf{p}},\emph{\textbf{w}}) + \frac{2}{d^2} \sum_{j < k} \left(p^{-1} - p_j^{-1} - p_k^{-1} + \frac{p_{jk}}{p_jp_k}\right) \sigma_{jk}(\emph{\textbf{w}}) \\ &-\frac{2}{d} \sum_{j=1}^d(p^{-1} - p_j^{-1}) \sigma_j^{(1)}(\emph{\textbf{w}}) + \frac{2}{d} \sum_{j=1}^d \sum_{k=1}^d \left(p_k^{-1} - \frac{p_{jk}}{p_jp_k}\right) \sigma_{jk}^{(2)}(\emph{\textbf{w}}),
                \end{align*}
                and
                \begin{align*}
                    \mathcal{S}^{\mathcal{H}*}(\emph{\textbf{p}}, \emph{\textbf{w}}) =& \frac{1}{d^2} \sum_{j=1}^d (p^{-1} - p_j^{-1}) (1+\lambda_j(\emph{\textbf{w}})(d-1))^2 \sigma^2_j(\emph{\textbf{w}}) + \sigma_{d+1}^2(\emph{\textbf{p}},\emph{\textbf{w}}) \\ &+ \frac{2}{d^2} \sum_{j < k} \left(p^{-1} - p_j^{-1} - p_k^{-1} + \frac{p_{jk}}{p_jp_k}\right) (1+\lambda_j(\emph{\textbf{w}})(d-1))(1+\lambda_k(\emph{\textbf{w}})(d-1)) \sigma_{jk}(\emph{\textbf{w}}) \\ &-\frac{2}{d} \sum_{j=1}^d(p^{-1} - p_j^{-1}) (1+\lambda_j(\emph{\textbf{w}})(d-1)) \sigma_j^{(1)}(\emph{\textbf{w}}) \\ &+ \frac{2}{d} \sum_{j=1}^d \sum_{k=1}^d\left(p_k^{-1} - \frac{p_{jk}}{p_jp_k}\right) (1+\lambda_j(\emph{\textbf{w}})(d-1))\sigma_{jk}^{(2)}(\emph{\textbf{w}}),
                \end{align*}
                where explicit expressions of the functions $\sigma_j^2$ for $j \in \{1,\dots,d\}$, $\sigma_{d+1}^2$, $\sigma_{jk}$ with $j < k$, $\sigma_j^{(1)}$ with $j \in \{1,\dots,d\}$, $\sigma_{jk}^{(2)}$ for $j, k \in \{1,\dots,d\}$ are detailed in the proof for the sake of readibility.
            \end{proposition}
            Considering the special case of independent copula, Corollary \ref{coro_independent} below gives a closed form of the limit variance which no longer depends on the Pickands dependence function.
            \begin{corollary}
                \label{coro_independent}
                In the framework of Theorem \ref{weak_conv_hybrid_mado} and if $C(\textbf{u}) = \Pi_{j=1}^d u_j$, then the functions $\sigma_{d+1}^2$, $\sigma_j^{(1)}$ with $j \in \{1,\dots,d\}$, have the following forms, for $\emph{\textbf{w}} \in \Delta^{d-1}$ :
                \begin{align*}
                    \sigma_{d+1}^2(\emph{\textbf{p}}, \emph{\textbf{w}}) &= \frac{1}{4} \left(\frac{1}{3p} - \sum_{j=1}^d p_j^{-1} \frac{w_j}{4-w_j} \right), \\
                    \sigma_j^{(1)}(\emph{\textbf{w}}) &= \frac{1}{2}\left[ \frac{1}{3} - \frac{1}{1+w_j} \right] + \frac{w_j}{3(1+w_j)(3+w_j)},
                \end{align*}
                and $\sigma_{jk}$ for $j < k$, $\sigma_{jk}^{(2)}$ for $j < k $ and $\sigma_{kj}^{(2)}$ with $k < j$ are constants and equal to $0$.
            \end{corollary}
            \begin{remark}
                From our knowledge, only \cite{GuiNa} gave an explicit value of the variance for the madogram of a bivariate random vector considering the independent copula. The result stated in Corollary \ref{coro_independent} is not an extension of this result because the hypothesis $\emph{\textbf{w}} \in \Delta^{d-1}$ is crucial. Nevertheless, the same techniques used to prove Proposition \ref{Boulin} can be applied to show a similar explicit formula of the asymptotic variance for an extension of the madogram in \cite{GuiNa} for  $d \geq 2$.
            \end{remark}
            Weak consistency of our estimators directly comes down from Proposition \ref{Boulin}. We are nonetheless able to state the strong consistency only under  Assumption \ref{Cond_2}.
            \begin{proposition}[\textbf{Strong consistency}]
                \label{strong_consistency}
                Let $(\textbf{I}_i, \tilde{\textbf{X}_i})_{i=1}^n$ an i.i.d sample given by Equation (\ref{missing_2}). Under Assumption \ref{Cond_2} for a fixed $\emph{\textbf{w}} \in \Delta^{d-1}$, it holds that 
                \begin{equation*}
                    \hat{\nu}_n^\mathcal{H}(\emph{\textbf{w}}) \overunderset{a.s.}{n \rightarrow \infty}{\rightarrow} \nu(\emph{\textbf{w}}), \quad \hat{\nu}_n^\mathcal{H*}(\emph{\textbf{w}}) \overunderset{a.s.}{n \rightarrow \infty}{\rightarrow} \nu(\emph{\textbf{w}}).
                \end{equation*}
            \end{proposition}
            For the rest of this section, we use our previous  results to state some properties of the Pickands estimator in the missing data framework. 
            
            It is a common knowledge that the $\textbf{w}$-madogram is of main interest to construct of the Pickands dependence function. Indeed, given Equation \eqref{w_mado_pickands}, one can define an estimator of the Pickands dependence function by estimating the $\textbf{w}$-madogram and using it as a plug-in estimator. Most interesting properties of the $\textbf{w}$-madogram such as strong consistency and the weak convergence are thus translated for the Pickands estimator using continuous mapping theorem and the Delta method. In the missing data framework we define the following estimator.
            \begin{definition}
                \label{def:pick_hat}
                Let $(\textbf{I}_i, \tilde{\textbf{X}}_i)_{i=1}^n$ be a samble given by \eqref{missing_2}, the hybrid nonparametric estimator of the Pickands dependence function is defined as
                \begin{equation}
                    \label{hybrid_pickands_estimator}
                    \hat{A}_n^{\mathcal{H}*}(\textbf{w}) = \frac{\hat{\nu}_n^{\mathcal{H}*}(\textbf{w}) + c(\textbf{w})}{1-\hat{\nu}_n^{\mathcal{H}*}(\textbf{w}) - c(\textbf{w})},
                \end{equation}
                where $\hat{\nu}_n^{\mathcal{H}*}(\textbf{w})$ defined in Equation \eqref{corrected_lambda_FMado_hybrid} and $c(\textbf{w}) = d^{-1} \sum_{j=1}^d w_j / (1+w_j)$.
            \end{definition}
            Using the results of \cite{MARCON20171} (namely, Theorem 2.4), Proposition \ref{Boulin} and Proposition \ref{strong_consistency} of this paper, we state the following corollary.
            \begin{corollary}
                \label{asymptotic_distribution_pickands}
                    Let $\emph{\textbf{p}} = (p_1, \dots, p_d, p)$ and $(\textbf{I}_i, \tilde{\textbf{X}_i})_{i=1}^n$ be a samble given by \eqref{missing_2}. For $\emph{\textbf{w}} \in \Delta^{d-1}$, if $C$ is an extreme value copula with Pickands dependence function and  under Assumption \ref{Cond_2}, it holds that 
                \begin{equation*}
                    \hat{A}_n^{\mathcal{H}*}(\emph{\textbf{w}}) \overunderset{a.s.}{n \rightarrow \infty}{\rightarrow} A(\textbf{w}).
                \end{equation*}
                Furthermore, if $C$ additionally verifies Assumptions  \ref{Cond_1}.1 and \ref{Cond_1}.2, we obtain
                \begin{equation*}
                    \sqrt{n} \left(\hat{A}_n^{\mathcal{H}*}(\emph{\textbf{w}}) - A( \emph{\textbf{w}})\right) \overunderset{d}{n \rightarrow \infty}{\rightarrow} \mathcal{N}\left(0, \mathcal{V}(\emph{\textbf{p}}, \emph{\textbf{w}})\right),
                \end{equation*}
                where  the closed formula of the asymptoptic variance is given by $
                    \mathcal{V}(\emph{\textbf{p}}, \emph{\textbf{w}}) = (1+A(\emph{\textbf{w}}))^4 \mathcal{S}^{\mathcal{H}*}(\emph{\textbf{p}}, \emph{\textbf{w}}),$ 
                with $\mathcal{S}^{\mathcal{H}*}(\emph{\textbf{p}}, \emph{\textbf{w}})$ as in  Proposition \ref{Boulin}. 
            \end{corollary}
            
    \section{Numerical results}
            \label{numerical_results}
        In this section we verify our findings concerning the closed formula of the asymptotic variances through a simulation study. To do so, we compare empirical counterparts of the asymptotic variances computed out with Monte Carlo simulations with the explicit asymptotic variances given by Proposition \ref{Boulin}. Our simulation studies are implemented using \texttt{Python} programming language and all the codes are available online in this \href{https://github.com/Aleboul/missing}{github} repository.
        
        \subsection{Presentation of the models}
        
        We present here the six models (\ref{sym_log_model} to \ref{mod:Student}) used for this simulation study. The  $d$-dimensional Gumbel and the asymmetric logistic models are considered in models \ref{sym_log_model} and \ref{asy_log_model} below, the remaining ones (models \ref{asy_neg_log_model} to \ref{mod:Student}) concern only the bivariate case.
        \begin{enumerate}[label = \textbf{M\arabic*}, leftmargin=*]
            \item \textbf{The symmetric logistic, or Gumbel model} \cite{gumbel1960} is defined by the   following Pickands dependence function
            \begin{equation*}
                A(w_1,\dots,w_d) = \left( \sum_{j=1}^d w_j^\theta\right)^{1/\theta},
            \end{equation*}
            with $\theta \in [1, \infty)$.  We retrieve the independent case when $\theta = 1$ and the dependence between the variables is stronger as $\theta$ goes to infinity. The restriction to $d = 2$ is immediate from the definition. \label{sym_log_model}
            
            \item Let $B$ be the set of all nonempty subsets of $\{1,\dots,d\}$ and $B_1 = \{ b \in B, |b| = 1\}$, where $|b|$ denotes the number of elements in the set $b$. \textbf{The asymmetric logistic model} in  \cite{tawn1990} is defined by the following Pickands dependence function
            \begin{equation*}
                A(w_1, \dots, w_d) = \sum_{b \in B} \left( \sum_{j \in b} (\theta_{j,b} w_j)^{\theta_b}\right)^{1/\theta_b},
            \end{equation*}
            where $\theta_b \in [1,\infty)$ for all $b \in B \setminus B_1$, and the asymmetry parameters $\theta_{j,b} \in [0,1]$ for all $b \in B$ and $j \in b$. The model should verify the following constrains $\sum_{b \in B(j)} \theta_{j,b} = 1$ for $j \in \{1,\dots,d\}$ where $B_{(j)} = \{b \in B, j \in b\}$ and if $\theta_b = 1$ for every $b \in B \setminus B_1$, then $\theta_{j,b} = 0$ $\forall j \in b$. The model contains $2^d - d -1$ dependence parameters and $d(2^{d-1} -1)$ asymmetry parameters. In case of $d = 2$, we go back to the asymmetric logistic model in  \cite{10.1093/biomet/75.3.397}, namely
            \begin{equation*}
                A(w) = (1-\psi_1)w + (1-\psi_2)(1-w) + \left[ (\psi_1 w)^\theta + (\psi_2(1-w))^\theta \right]^{1/\theta},
            \end{equation*}
            with $\theta \in [1,\infty)$, $\psi_1, \psi_2 \in [0,1]$. For $d=3$, the Pickands dependence function is expressed as 
            \begin{align*}
                A(\textbf{w}) =& \alpha_1 w_1 + \psi_1 w_2 + \phi_1 w_3 + \left( (\alpha_2 w_1)^{\theta_1} + (\psi_2w_2)^{\theta_1} \right)^{1/\theta_1} + \left( (\alpha_3 w_2)^{\theta_2} + (\phi_2w_3)^{\theta_2} \right)^{1/\theta_2} \\ &+ \left( (\psi_3 w_2)^{\theta_3} + (\phi_3w_3)^{\theta_3} \right)^{1/\theta_3} 
                              + \left( (\alpha_4 w_1)^{\theta_4} + (\psi_4 w_2)^{\theta_4} + (\phi_4 w_3)^{\theta_4} \right)^{1/\theta_4},
            \end{align*}
            where $\boldsymbol{\alpha} = (\alpha_1, \dots, \alpha_4), \boldsymbol{\psi} = (\psi_1, \dots, \psi_4), \boldsymbol{\phi} = (\phi_1, \dots, \phi_4)$ are all elements of $\Delta^3$.\label{asy_log_model}
            
            \item \textbf{The asymmetric negative logistic model} in  \cite{Joe1990FamiliesOM} is  defined via 
            \begin{equation*}
               A(w) = 1- \left[ (\psi_1(1-w))^{-\theta} + (\psi_2 w)^{-\theta}\right]^{-1/\theta},
            \end{equation*}
            with parameters $\theta \in (0,\infty)$, $\psi_1, \psi_2 \in (0,1]$. The special case $\psi_1 = \psi_2 = 1$ returns the Galambos model \cite{Oliveira1977TheAT}. \label{asy_neg_log_model}
            
            \item \textbf{The asymmetric mixed model} in  \cite{10.1093/biomet/75.3.397} corresponds to
            \begin{equation*}
                A(w) = 1 - (\theta + \kappa)w + \theta w^2 + \kappa w^3,
            \end{equation*}
            with parameters $\theta$ and $\kappa$ satisfying $\theta \geq 0$, $\theta + 3\kappa \geq 0$, $\theta + \kappa \leq 1$, $\theta + 2\kappa \leq 1$. The special case $\kappa = 0$ and $\theta \in [0,1]$ yields the symmetric mixed model. In the symmetric mixed model, when $\theta = 0$, we recover the independent copula. \label{mixed_model}
            \item \textbf{The model of Hüsler and Reiss} in \cite{HUSLER1989283} is given by the Pickands dependence function  
            \begin{equation*}
                A(t) = (1-t) \Phi\left( \theta + \frac{1}{2\theta} log\left(\frac{1-t}{t}\right)\right) + t\Phi\left( \theta + \frac{1}{2\theta} log\left( \frac{t}{1-t} \right)\right),
            \end{equation*}
            where $\theta \in (0, \infty)$ and $\Phi$ is the standard normal distribution function. As $\theta$ goes to $0^+$, the dependence between the two variables increases. When $\theta$ goes to infinity, we are in case of near independence. \label{mod:Husler_Reiss}
            
            \item \textbf{The Student $t$-EV model} in \cite{Demarta_Mcneil} is given by 
            \begin{align*}
                &A(w) = wt_{\nu + 1}(z_w) + (1-w)t_{\nu+1}(z_{1-w}), \\
                &\textrm{with }  z_w = (1+\nu)^{1/2}[\{w/(1-w)\}^{1/\nu} - \theta](1-\theta^2)^{-1/2},
            \end{align*}
            and parameters $\nu > 0$, and $\theta \in (-1,1)$, where $t_{\nu+1}$ is the distribution function of a Student-$t$ random variable with $\nu+1$ degrees of freedom. \label{mod:Student}
        \end{enumerate}
        \subsection{Description of numerical experiments}
        For each numerical experiment, the endpoint-corrected $\textbf{w}$-madogram estimator in \eqref{corrected_lambda_FMado_hybrid} is computed using $\lambda_j(\textbf{w}) = w_j$. The study consists in three different experiments (\textbf{E1}, \textbf{E2} and \textbf{E3}). For all experiments, the empirical counterpart of the asymptotic variance given by Proposition \ref{Boulin} is computed out through a given grid of the simplex $\Delta^{d-1}$. For a given element $\textbf{w}$ of this grid, $n_{iter} \in \mathbb{N} \setminus \{0\}$ random samples of size $n$ are generated from the models \ref{sym_log_model} to \ref{mod:Student} given above. By using these samples we estimate the  associated $\textbf{w}$-madogram. We thus compute the empirical variance of the normalized estimation error namely,
        \begin{equation}
        \label{empirical_variance}
            \mathcal{E}_n^{\mathcal{H}}(\textbf{w}) \triangleq \widehat{Var}\left( \sqrt{n}\left(\boldsymbol{\hat{\nu}_n^{\mathcal{H}}}(\textbf{w}) - \nu(\textbf{w})\right)\right), \quad \mathcal{E}_n^{\mathcal{H}*}(\textbf{w}) \triangleq \widehat{Var}\left( \sqrt{n}\left(\boldsymbol{\hat{\nu}_n^{\mathcal{H}*}}(\textbf{w}) - \nu(\textbf{w})\right)\right),
        \end{equation}
        where $\boldsymbol{\hat{\nu}_n^{\mathcal{H}}}$ and $\boldsymbol{\hat{\nu}_n^{\mathcal{H}*}}$ are the vectors composed out of the $n_{iter}$ hybrid and corrected estimators (see Equations \eqref{hybrid_lambda_fmado} and \eqref{corrected_lambda_FMado_hybrid}) of the $\textbf{w}$-madogram, respectively. We also define the Mean Integrated Squared Error (MISE) between $\mathcal{E}_n^{\mathcal{H}}$ and $\mathcal{S}^{\mathcal{H}}$ the asymptotic variance computed in Proposition \ref{Boulin} (\emph{resp.} between $\mathcal{E}_n^{\mathcal{H}*}$ and $\mathcal{S}^{\mathcal{H}*}$), that is
        \begin{equation}
            \label{eq:mise}
                MISE^{\mathcal{H}} \triangleq \mathbb{E}\left[ \int_{\Delta^{d-1}} \left( \mathcal{E}_n^{\mathcal{H}}(\textbf{w}) - \mathcal{S}^{\mathcal{H}}(\textbf{p},\textbf{w}) \right)^2 d\textbf{w} \right], \quad\quad  MISE^{\mathcal{H}*} \triangleq \mathbb{E}\left[ \int_{\Delta^{d-1}} \left( \mathcal{E}_n^{\mathcal{H}*}(\textbf{w}) - \mathcal{S}^{\mathcal{H}*}(\textbf{p},\textbf{w}) \right)^2 d\textbf{w} \right].
        \end{equation}
        \begin{itemize}
            \item[\textbf{E1}] We set $d = 2$. A Monte Carlo study is implemented here to illustrate Proposition \ref{Boulin} in finite-sample setting with missing data. We consider \ref{asy_log_model}, \ref{asy_neg_log_model}, \ref{mixed_model}, \ref{mod:Husler_Reiss} and \ref{mod:Student} where we fix $n_{iter} = 300$ and $n = 1024$. The chosen grid is $\{1/200, \dots, 199/200 \}$ and we take $p_1 = p_2 = 0.75$. We estimate $MISE^{\mathcal{H}}$ in \eqref{eq:mise} by
            \begin{equation*}
                \widehat{MISE}^{\mathcal{H}}_n = \frac{1}{10} \sum_{l=1}^{10}\frac{1}{199}\sum_{k=1}^{199} \left( \mathcal{E}_{n,l}^{\mathcal{H}}\left(\frac{k}{200}\right) - \mathcal{S}^{\mathcal{H}}\left(\textbf{p},\frac{k}{200}\right) \right)^2,
            \end{equation*}
            with $\mathcal{E}_{n,l}^{\mathcal{H}}, l \in \{1,\dots,10\}$ is the empirical counterpart of $\mathcal{S}^{\mathcal{H}}$ taking the empirical variance of $30$ estimators $\hat{\nu}_n^{\mathcal{H}}(\textbf{w})$ where $\textbf{w} = (k /200, 1- k/200)$ and $k \in \{1,\dots,199\}$. Each estimator of the $\textbf{w}$-madogram is computed out through a random sample with $n = 1024$. By using the second equation in  \eqref{eq:mise}, the estimator $\widehat{MISE}_n^{\mathcal{H}*}$ is defined similarly.
            \item[\textbf{E2}] We fix $d = 3$ and we consider \ref{sym_log_model} and \ref{asy_log_model} with $n_{iter} = 100$ and $n = 512$. We set the dependence parameter as $\theta = 1$ and $\theta = 2$ for the first model. For the second one we take $\boldsymbol{\alpha} = (0.4,0.3,0.1,0.2)$, $\boldsymbol{\psi} = (0.1, 0.2, 0.4, 0.3)$, $\boldsymbol{\phi} = (0.6,0.1,0.1,0.2)$ and $\boldsymbol{\theta} = (\theta_1, \dots, \theta_4) = (0.6,0.5,0.8,0.3)$ as the dependence parameter. We take $p_1 = p_2 = p_3 = 0.9$ and thus $p = 0.729$, $p_{ij} = 0.81$ with $i,j \in \{1,2,3\}$ and $i < j$. We grid the $[0,1]^2$ cube into $10000$ points at same distance from each other and we only keep those with $w_2 + w_3 < 1.0$ where $w_2$ and $w_3$ are in the grid of the cube, we set $w_1 = 1-w_2-w_3$. Let $\Delta_n^{d-1}$ be $199$ points uniformly sampled from $\Delta^2$ and $n_{iter} = 300$, Equation \eqref{eq:mise} is estimated with
            \begin{equation*}
            \widehat{MISE}^{\mathcal{H}}_n = \frac{1}{10}\sum_{l=1}^{10}\frac{1}{199}\sum_{k \in \Delta_n^{d-1}} \left( \mathcal{E}_{n,l}^{\mathcal{H}}\left(k\right) - \mathcal{S}^{\mathcal{H}}\left(\textbf{p},k\right) \right)^2,
            \end{equation*}
            where $\mathcal{E}_{n,l}^{\mathcal{H}}, l \in \{1,\dots,10\}$ is the empirical counterpart of $\mathcal{S}^{\mathcal{H}}$ taking the empirical variance of $30$ estimators $\hat{\nu}_n^{\mathcal{H}}(\textbf{w})$ with $\textbf{w} \in \Delta_n^{d-1}$. Each estimator of the $\textbf{w}$-madogram is computed out through a random sample with $n = 512$. Again, $\widehat{MISE}^{\mathcal{H}*}_n$ is defined in a similar way.
            \item[\textbf{E3}] In this experiment, we aim to show that our conclusions are verified in a high dimension setting. We compute empirical counterpart of the asymptotic variance for a varying dimension $d$ and we compare its value to the theoretical one given by Proposition \ref{Boulin}. Furthermore, as the probability of observing a complete row decrease quickly with respect to the dimension $d$, \emph{i.e.} $p = p_1^{-d}$, we set that there is no missing data. We consider the symmetric logistic model with dependence parameter $\theta = 2$. We sample $300$ points from the unit simplex $\Delta^{d-1}$ and we compute the following quantity
        \begin{equation}
            \label{rapport}
            \delta_n^{\mathcal{H}}(\textbf{w}) \triangleq \frac{\left|\mathcal{E}^{\mathcal{H}}_{n}(\textbf{w}) - \mathcal{S}^{\mathcal{H}}(\textbf{1}, \textbf{w})\right|}{\mathcal{S}^{\mathcal{H}}(\textbf{1}, \textbf{w})},
        \end{equation}
        where $\mathcal{E}^{\mathcal{H}}_n$ is computed from $n_{iter} = 100$ estimators of the $\textbf{w}$-madogram with sample size $n \in \{216,512,1024\}$. The results are collected for several values of $d \in \{5,10,\dots,40\}$.
        \end{itemize} 
        
        Note that for Experiments \textbf{E1} and \textbf{E2}, the missing mechanism is such as $I_1, \dots, I_d$ are pairwise independent and $p_j = p_1, \forall j\in \{1,\dots,d\}$. The independence setup corresponds to the worst scenario where the missingness of one variable does not influence the missingness of the other variables. \emph{A contrario}, if we suppose that $I_1, \dots, I_d$ are strongly dependent, \emph{i.e.} none or all entries are missing, we then estimate a statistic on a sample of average length $p \times n$ and we are turning back to inference in a complete data framework with a reduced sample size. This is also readily seen from the closed formula in Proposition \ref{Boulin}, indeed in a strongly dependent setting  we have $p = p_1$, so the asymptotic variance is reduced to the complete data framework up to a multiplicative factor.  
        
        \subsection{Results of experiments}
        Results of Experiment \textbf{E1} are depicted in Figure \ref{fig:missing_estimation}. For all panels, empirical counterparts given by Equation \eqref{empirical_variance} (points) fit the theoretical values exhibited from Proposition \ref{Boulin} (solid lines). For the hybrid estimator, as discussed in Remark \ref{remark_corr}, both empirical and theoretical values of the asymptotic variance are different from zero for each $w \in \{\{0\},\{1\}\}$. The corrected version provides  this feature and also modifies the shape of the curve (see Remark \ref{rem:remark_2}). Indeed  the asymptotic behavior of the hybrid and the corrected estimators are different in the missing data framework. Notice that, in terms of variance, we do not have a strict dominance from one estimator to another. 
        
        \begin{figure}[!htp]
                \begin{subfigure}{.33\textwidth}
                  \centering
                  \includegraphics[width=.9\linewidth,height=5.5cm]{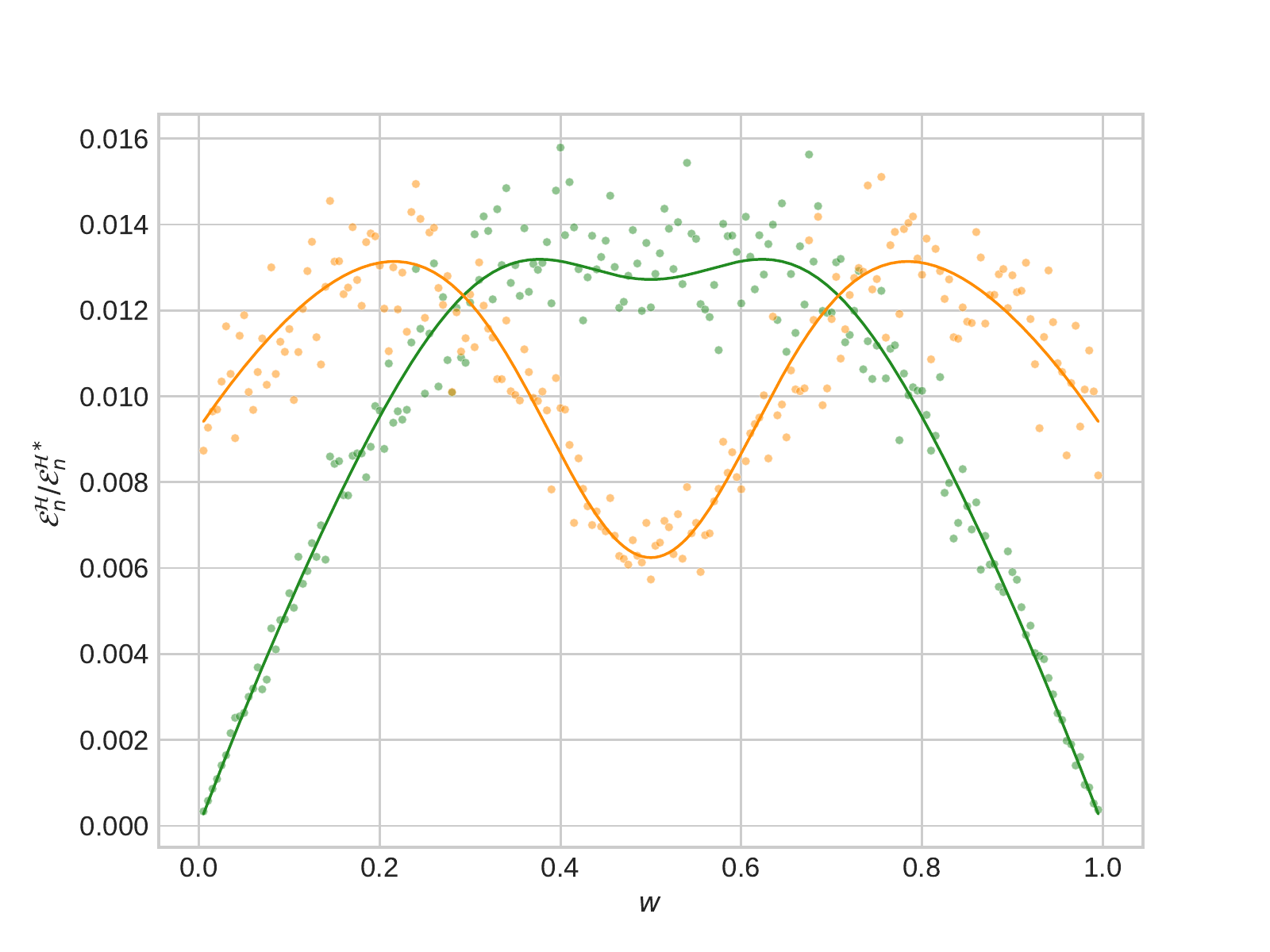}
                  \caption{\textbf{GAL (\ref{asy_neg_log_model}, $\theta = 2.5$})}
                  \label{fig:sfig1_chap2}
                \end{subfigure}%
                \begin{subfigure}{.33\textwidth}
                  \centering
                  \includegraphics[width=.9\linewidth,height=5.5cm]{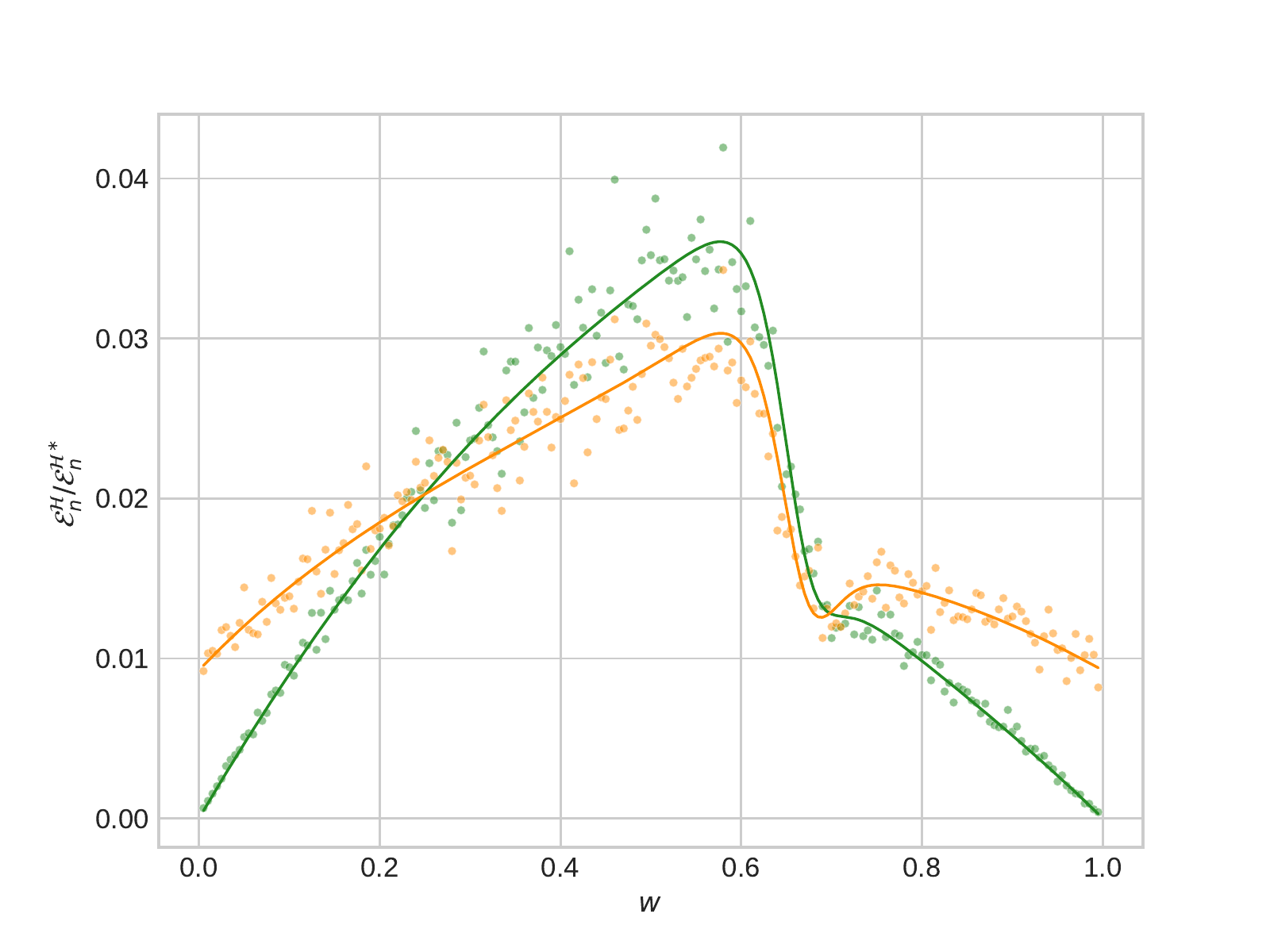}
                  \caption{\textbf{ANL (\ref{asy_neg_log_model}, $\theta = 10$, $\psi_1 = .5$, $\psi_2 = 1$)}}
                  \label{fig:sfig2_chap2}
                \end{subfigure}
                \begin{subfigure}{.33\textwidth}
                  \centering
                  \includegraphics[width=.9\linewidth,height=5.5cm]{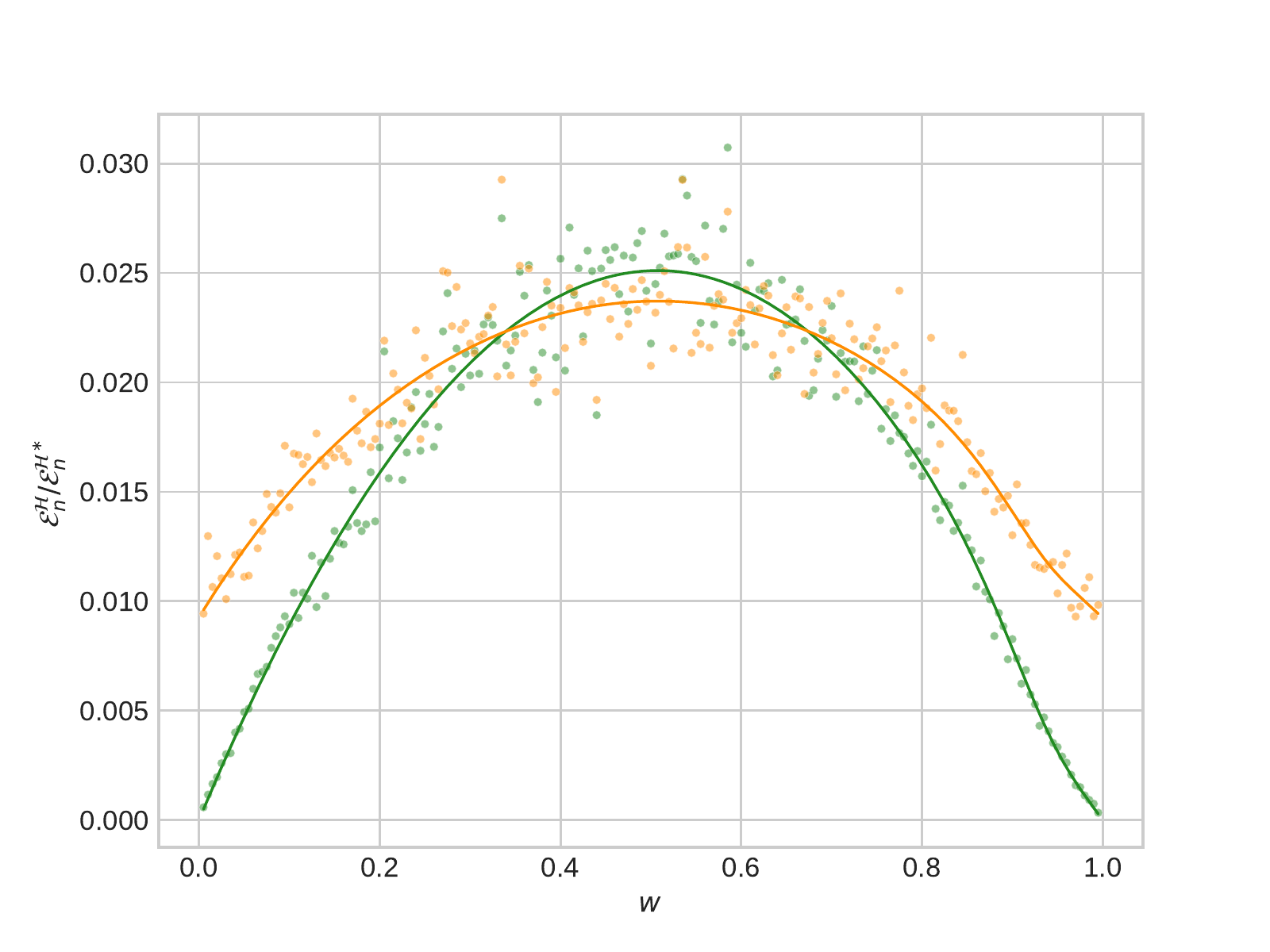}
                  \caption{\textbf{ASL (\ref{asy_log_model}, $\theta = \frac{5}{2}$, $\psi_1 = .1$, $\psi_2 = 1$)}}
                  \label{fig:sfig3_chap2}
                \end{subfigure}%
                \hspace{\fill}
                \begin{subfigure}{.33\textwidth}
                  \centering
                  \includegraphics[width=.9\linewidth,height=5.5cm]{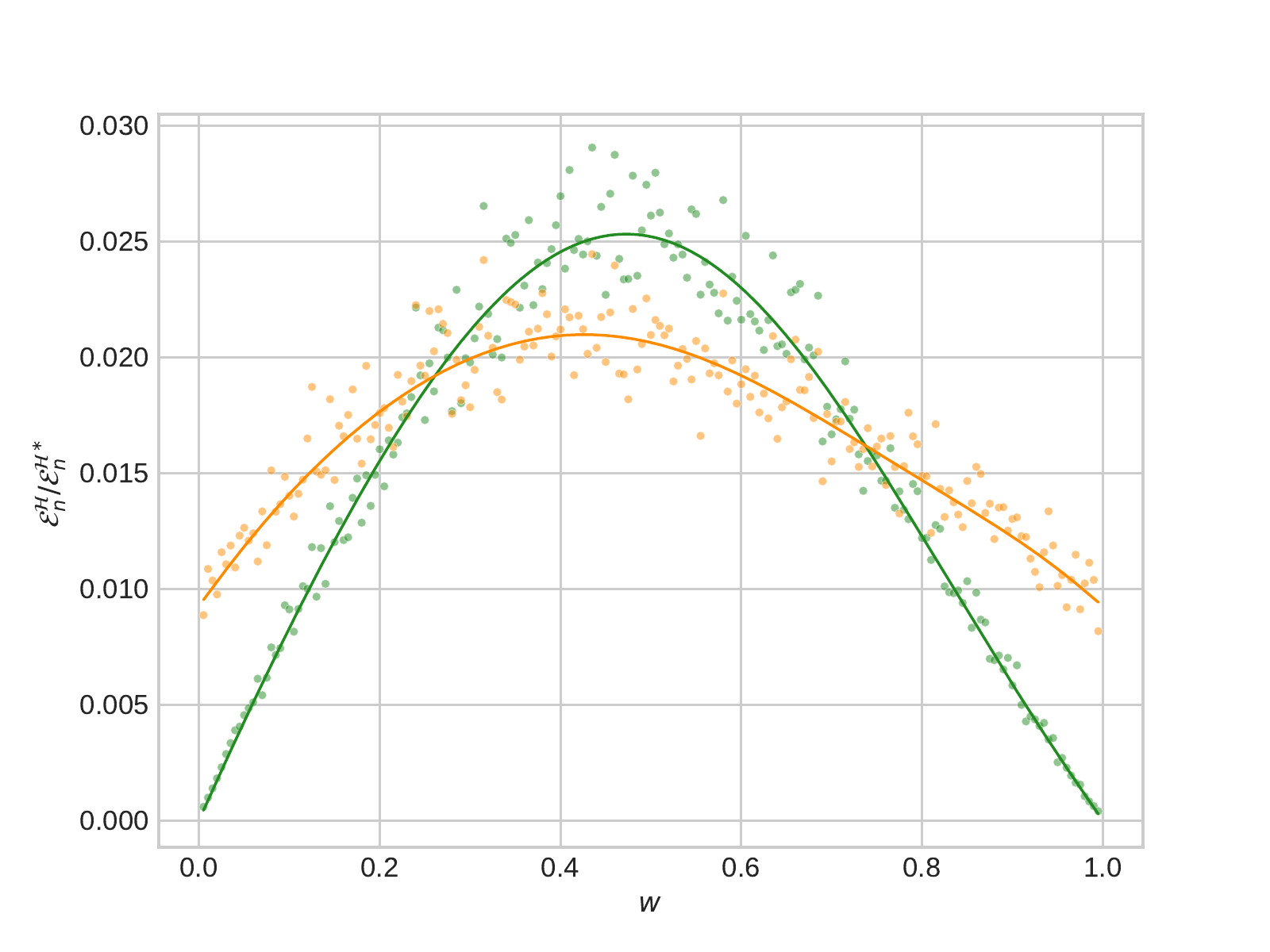}
                  \caption{\textbf{ASM (\ref{mixed_model}, $\theta = \frac{4}{3}$, $\kappa = -\frac{1}{3}$)}}
                  \label{fig:sfig4_chap2}
                \end{subfigure}
                \begin{subfigure}{.33\textwidth}
                  \centering
                  \includegraphics[width=.9\linewidth,height=5.5cm]{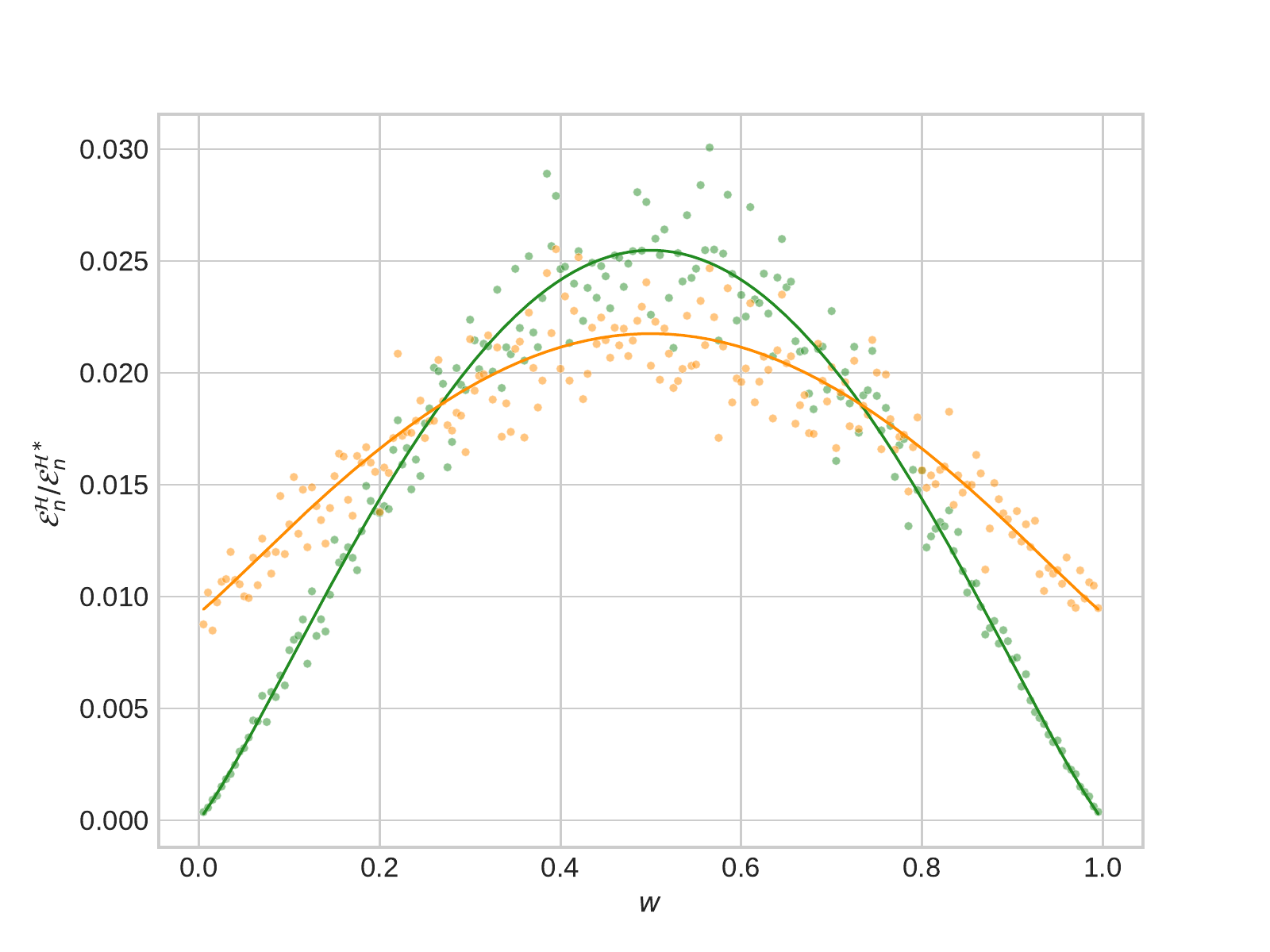}
                  \caption{\textbf{HR (\ref{mod:Husler_Reiss}, $\theta = 1.0$)}}
                  \label{fig:sfig5_chap2}
                \end{subfigure}
                \begin{subfigure}{.33\textwidth}
                  \centering
                  \includegraphics[width=.9\linewidth,height=5.5cm]{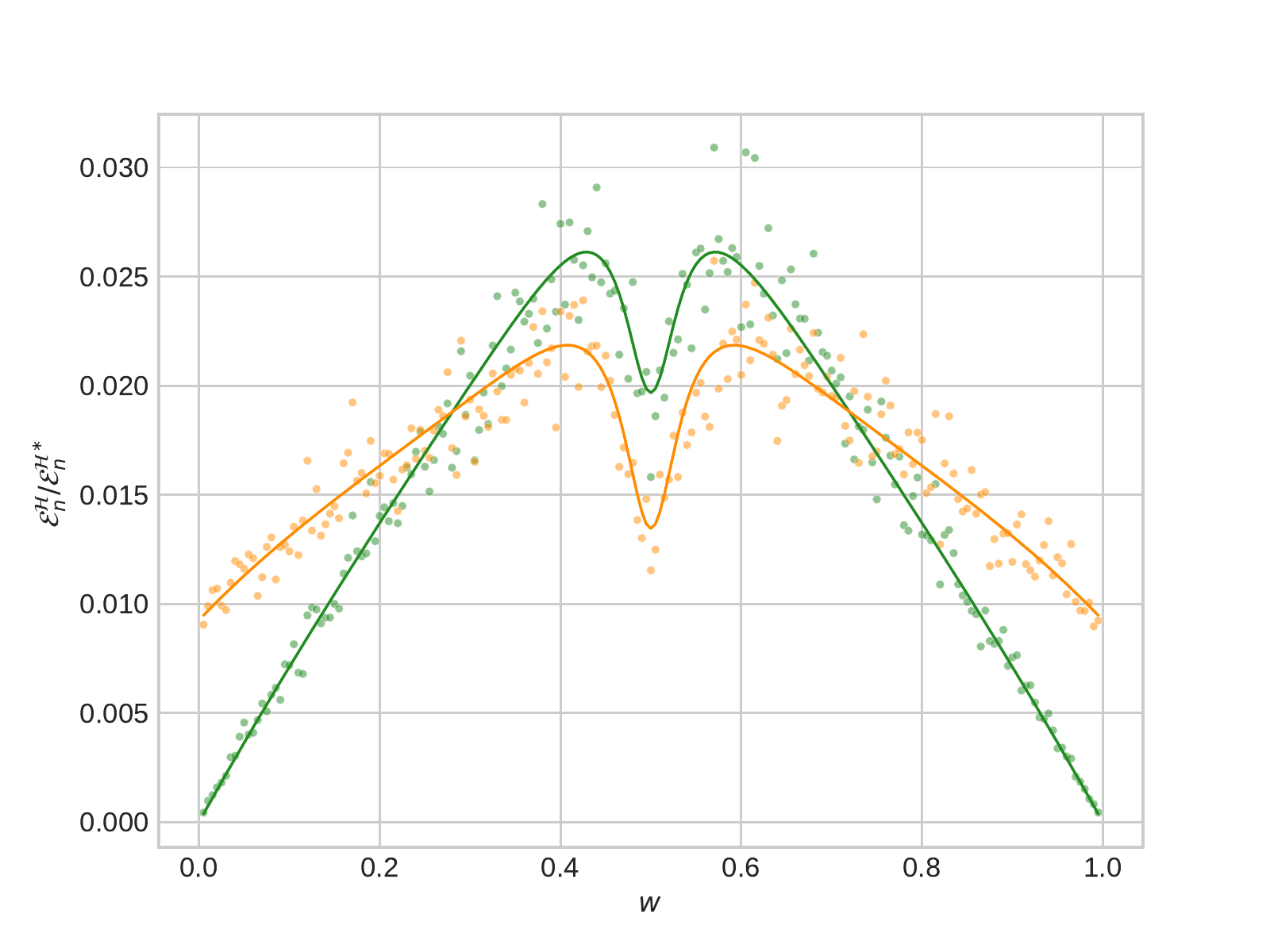}
                  \caption{\textbf{tEV (\ref{mod:Student}, $\theta = 0.8$, $\nu = 0.2$)}}
                  \label{fig:sfig6_chap2}
                \end{subfigure}
                \caption{$\mathcal{E}_n^\mathcal{H}$ in red and $\mathcal{E}_n^\mathcal{H*}$  in green  (see \eqref{empirical_variance})  as a function of $w$, of the asymptotic variances of the estimators of the $\textbf{w}$-madogram for six extreme-value copula models. The empirical variances are based on $300$ samples of size $n = 1024$. Solid lines are the theoretical value given by Proposition \ref{Boulin}.}
                \label{fig:missing_estimation}
        \end{figure}
    
        Results for Experiment \textbf{E2} are depicted in Figures \ref{fig:missing_estimation_exp2} and \ref{fig:missing_estimation_exp3}. In Figure \ref{fig:missing_estimation_exp2}, empirical counterparts given by Equation \eqref{empirical_variance} are depicted with points and closed expressions of the asymptotic variance  given by Proposition \ref{Boulin} are  drawn by a surface.  Figure \ref{fig:missing_estimation_exp3} presents the same studies differently by showing the level sets associated to the surfaces of Figure \ref{fig:missing_estimation_exp2}. As in Experiment \textbf{E1}, empirical counterparts given by the points fits the surface. Also, for the first row of Figure \ref{fig:missing_estimation_exp2}, we see that if $\textbf{w} \in \{ \{\textbf{e}_1\}, \{\textbf{e}_2\}, \{\textbf{e}_3\}\}$ then both theoretical and empirical counterparts are different from zero while this feature no longer applies in the second row with the introduction of the corrected version. In this two figures, we see that  $\mathcal{E}_n^\mathcal{H}$ and $\mathcal{E}_n^\mathcal{H*}$ and their empirical counterparts are close.
        \begin{figure}[!htp]
                \begin{subfigure}{.333\textwidth}
                  \centering
                  \caption{\textbf{IND (\ref{sym_log_model}, $\theta = 1.0$})}
                  \includegraphics[width=\linewidth]{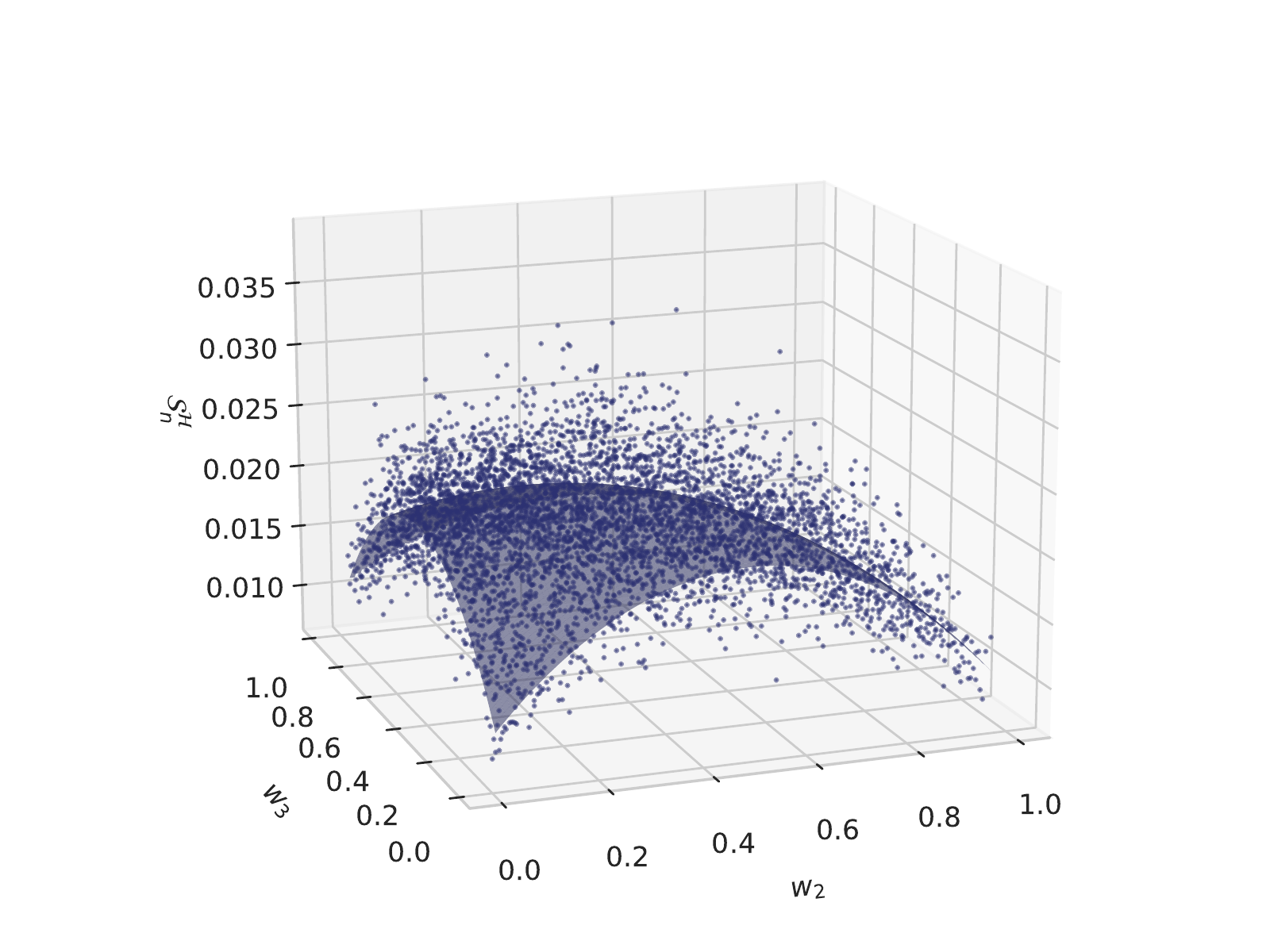}
                  \label{fig:sfig1_exp2}
                \end{subfigure}%
                \begin{subfigure}{.333\textwidth}
                  \centering
                  \caption{\textbf{LOG (\ref{sym_log_model}, $\theta = 2.0$})}
                  \includegraphics[width=\linewidth]{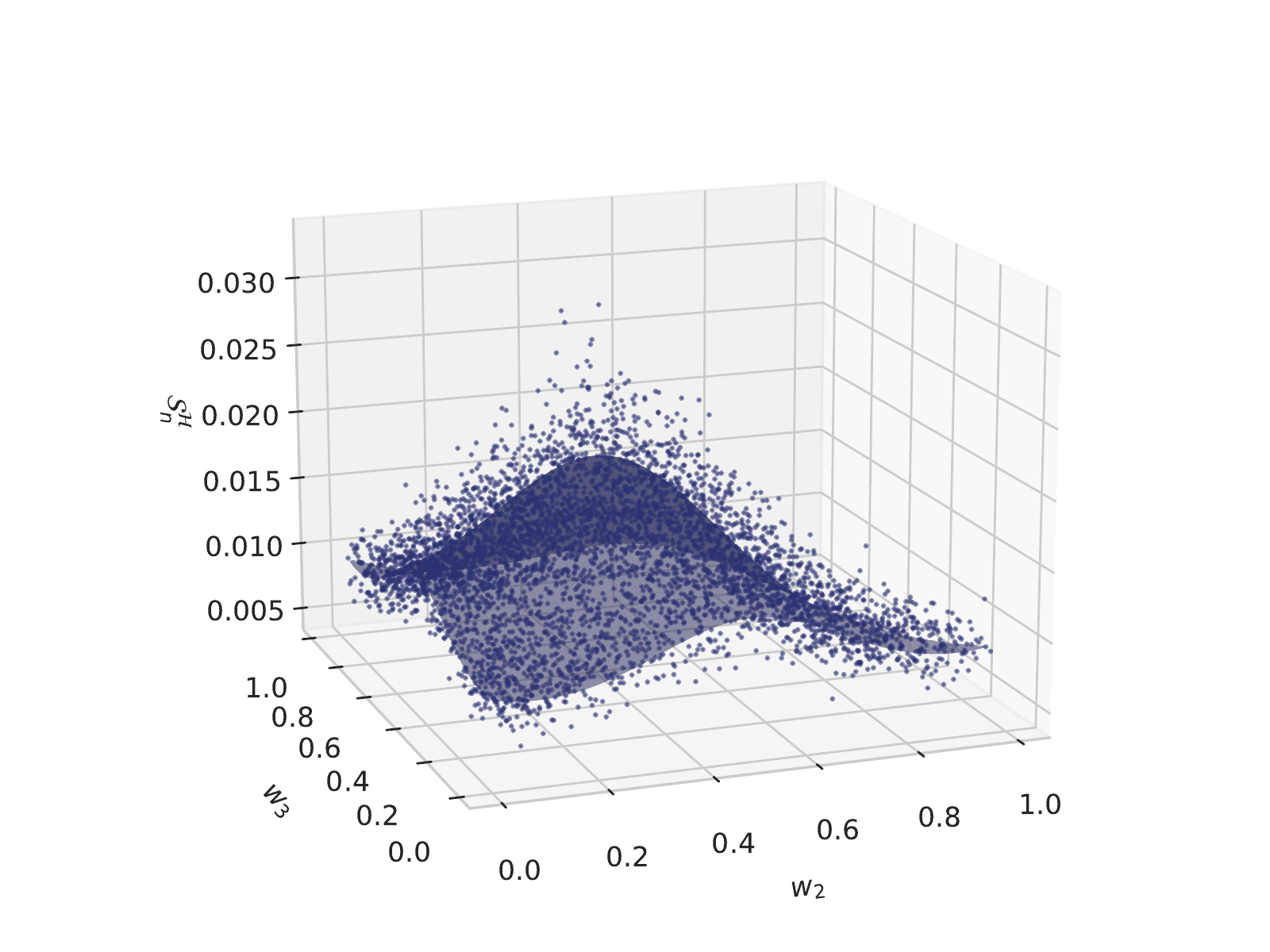}
                  \label{fig:sfig2_exp2}
                \end{subfigure}
                \begin{subfigure}{.333\textwidth}
                  \centering
                  \caption{\textbf{ASL (\ref{asy_log_model}})}
                  \includegraphics[width=\linewidth]{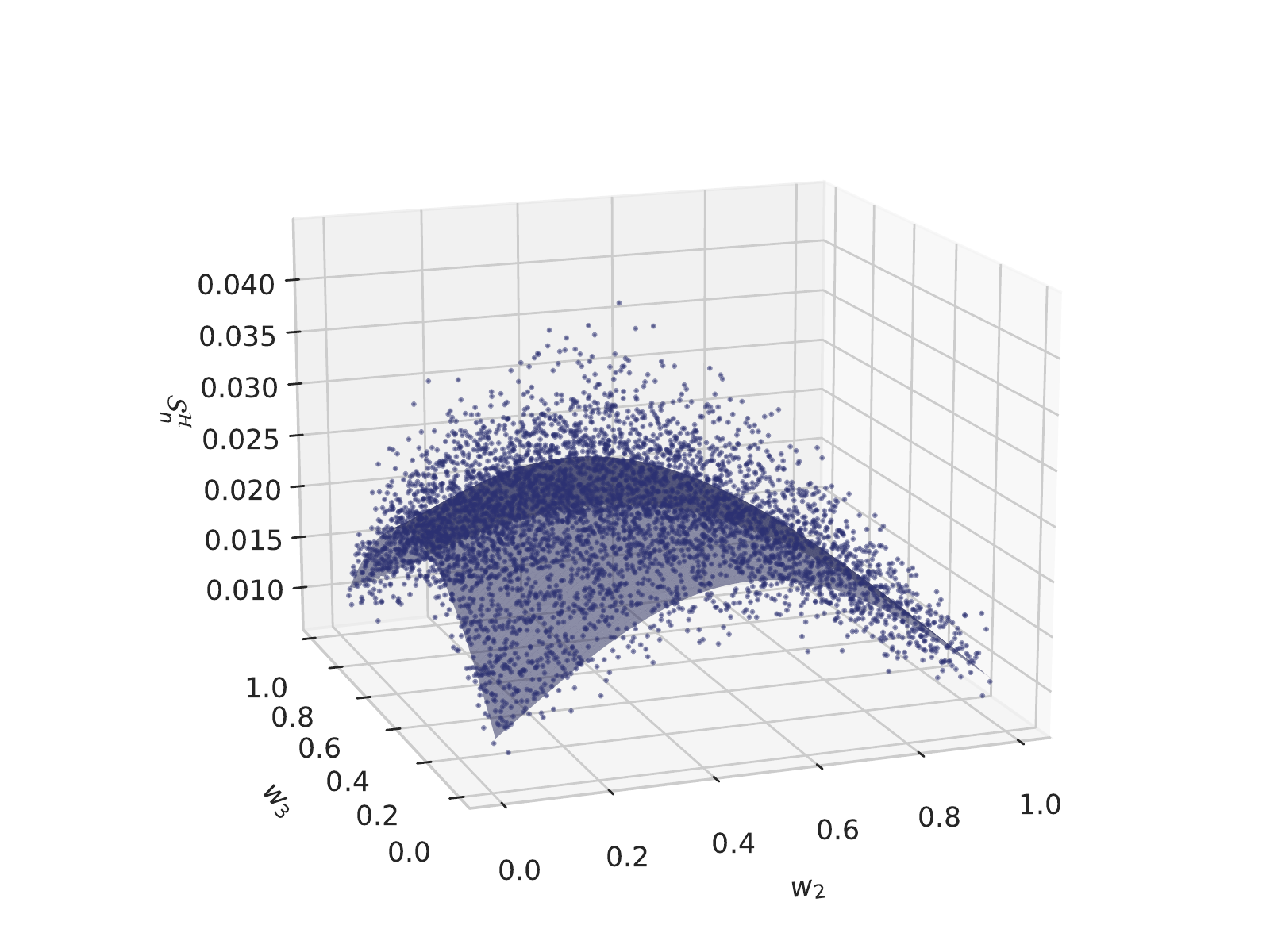}
                  \label{fig:sfig3_exp2}
                \end{subfigure}%
                \hspace{\fill}
                \begin{subfigure}{.333\textwidth}
                  \centering
                  \includegraphics[width=\linewidth]{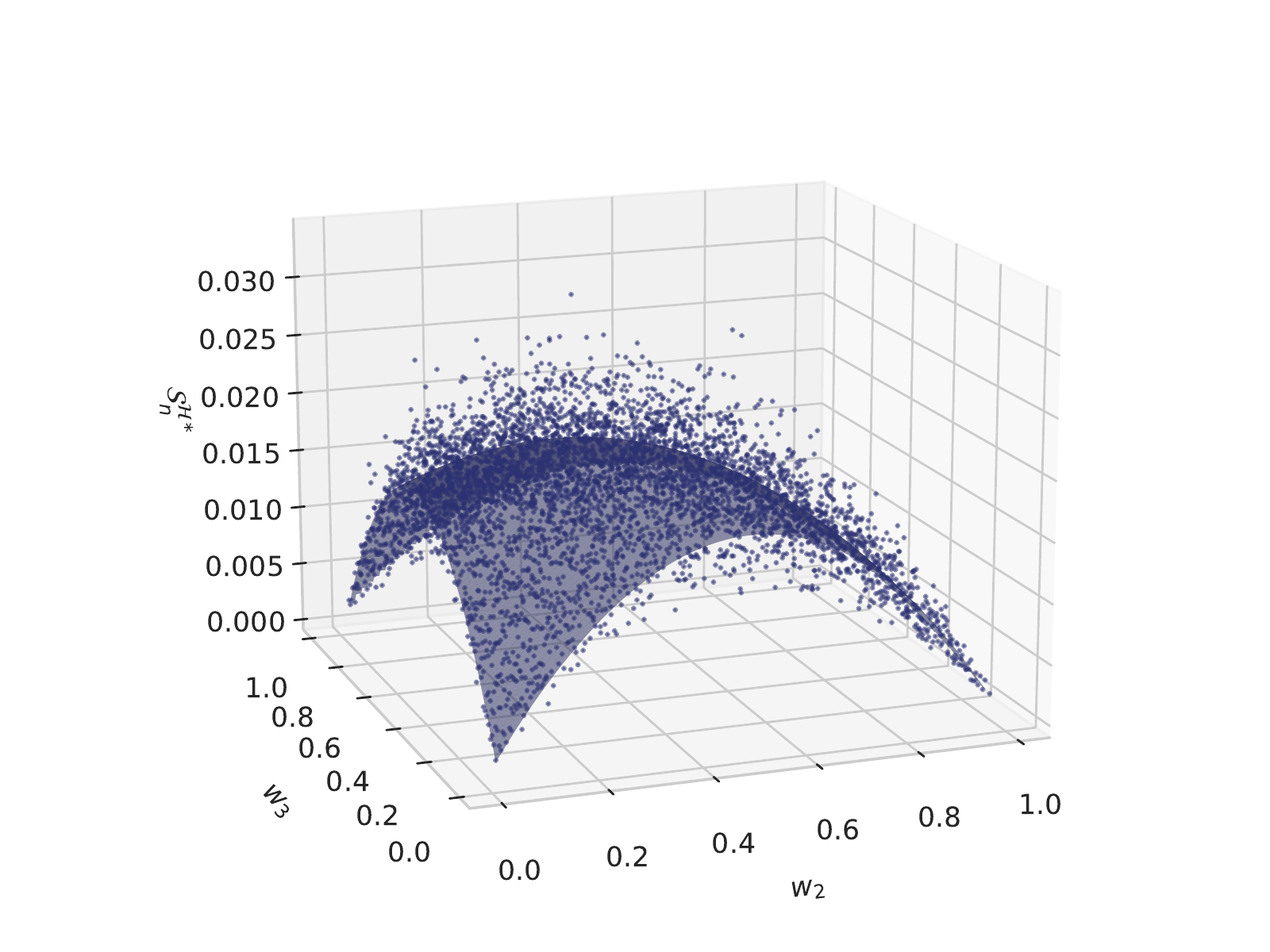}
                  \label{fig:sfig7_exp2}
                \end{subfigure}
                \begin{subfigure}{.333\textwidth}
                  \centering
                  \includegraphics[width=\linewidth]{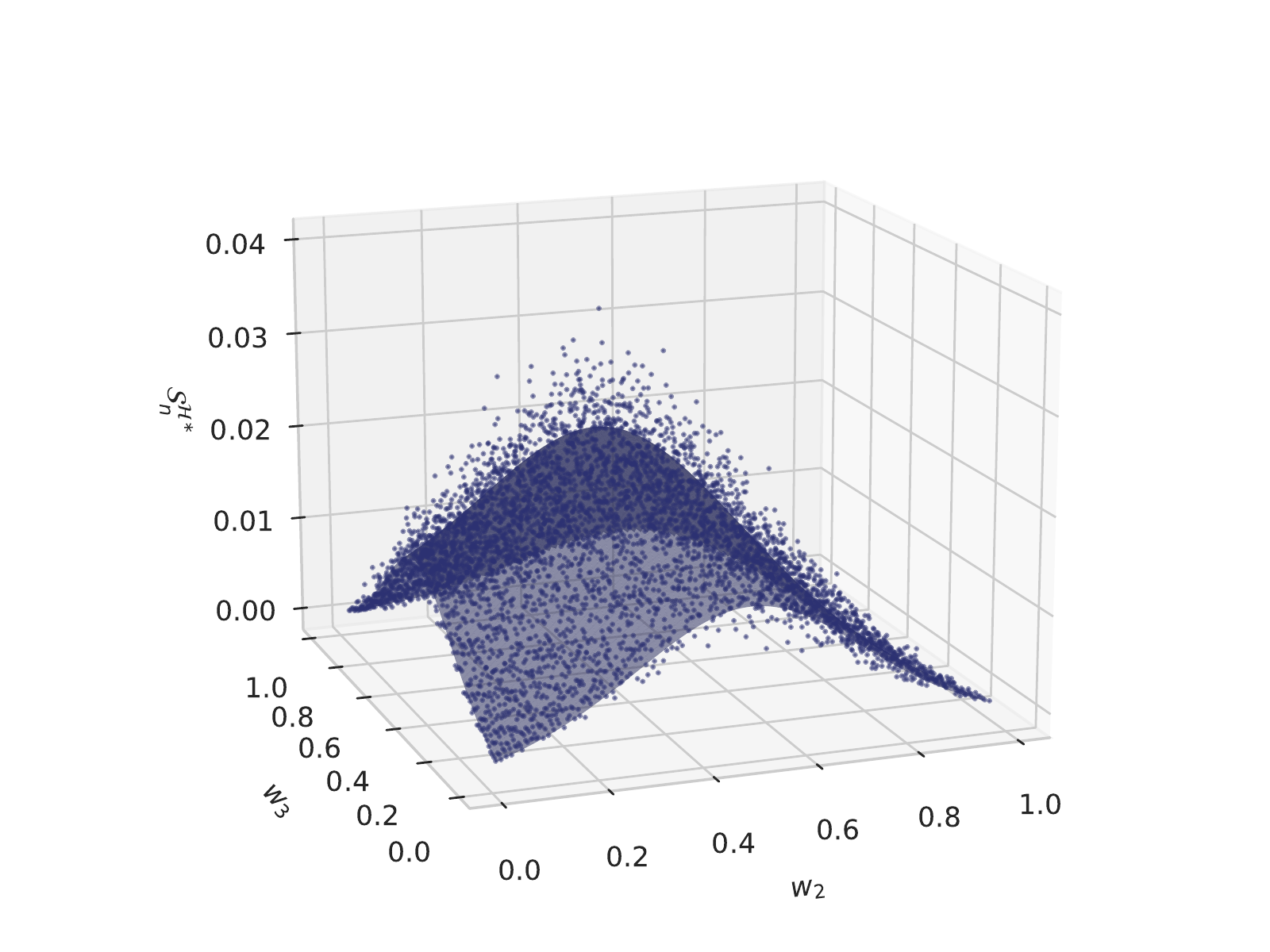}
                  \label{fig:sfig8_exp2}
                \end{subfigure}%
                \begin{subfigure}{.333\textwidth}
                  \centering
                  \includegraphics[width=\linewidth]{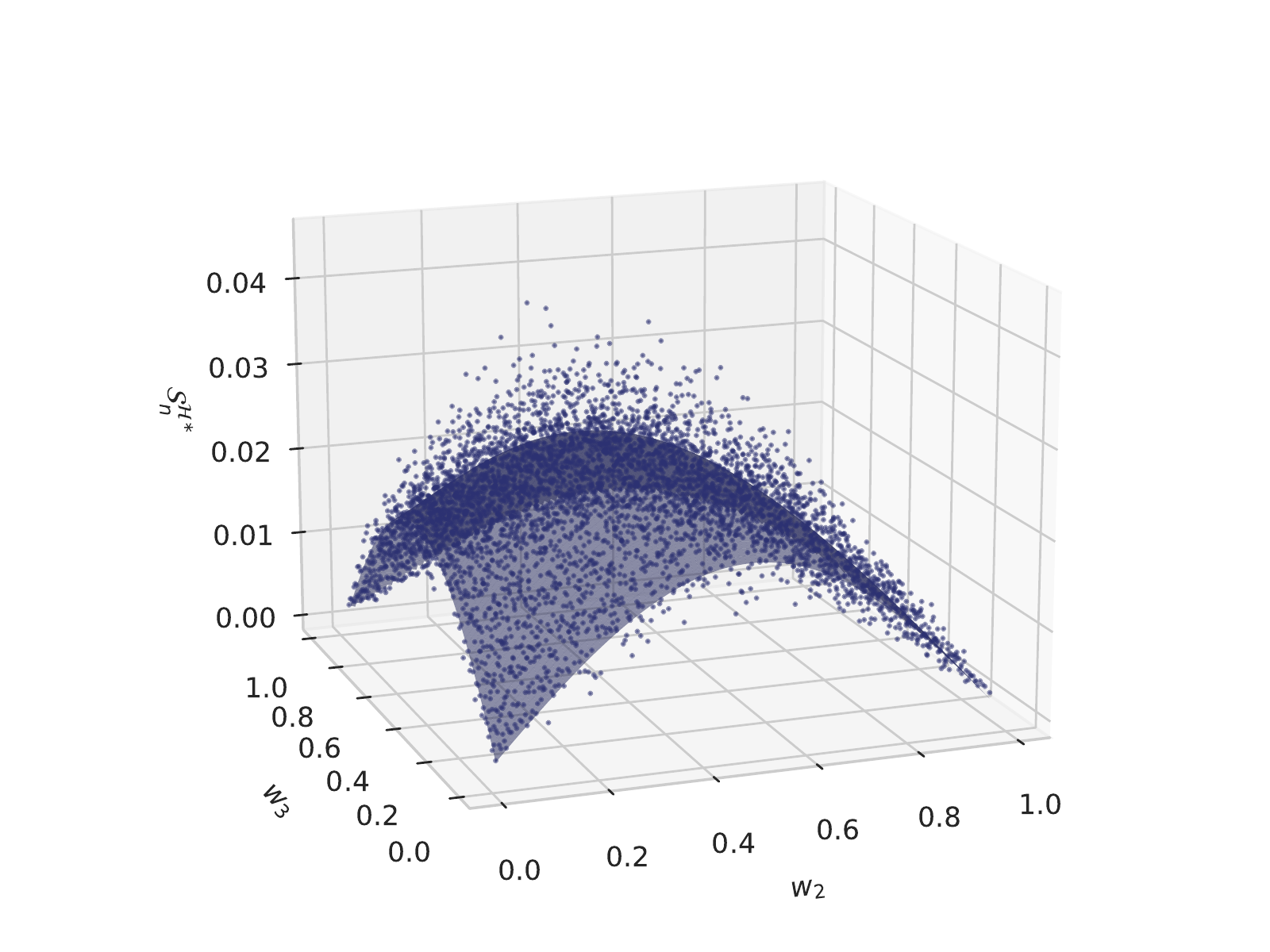}
                  \label{fig:sfig9_exp2}
                \end{subfigure}
                \hspace{\fill}
                \caption{$\mathcal{E}_n^\mathcal{H}$ (first row)  and $\mathcal{E}_n^\mathcal{H*}$ (second row) given by \eqref{empirical_variance}  as a function of $\textbf{w}$-madogram. The empirical variances are based on $100$ samples of size $n = 512$. Empirical counterparts are represented with points and theoretical values given by Proposition \ref{Boulin} are drawn by a surface.} 
                \label{fig:missing_estimation_exp2}
        \end{figure}
        
        \begin{figure}[!htp]
            \begin{subfigure}{.333\textwidth}
                  \centering
                  \caption{\textbf{IND (\ref{sym_log_model}, $\theta = 1.0$})}
                  \includegraphics[width=\linewidth]{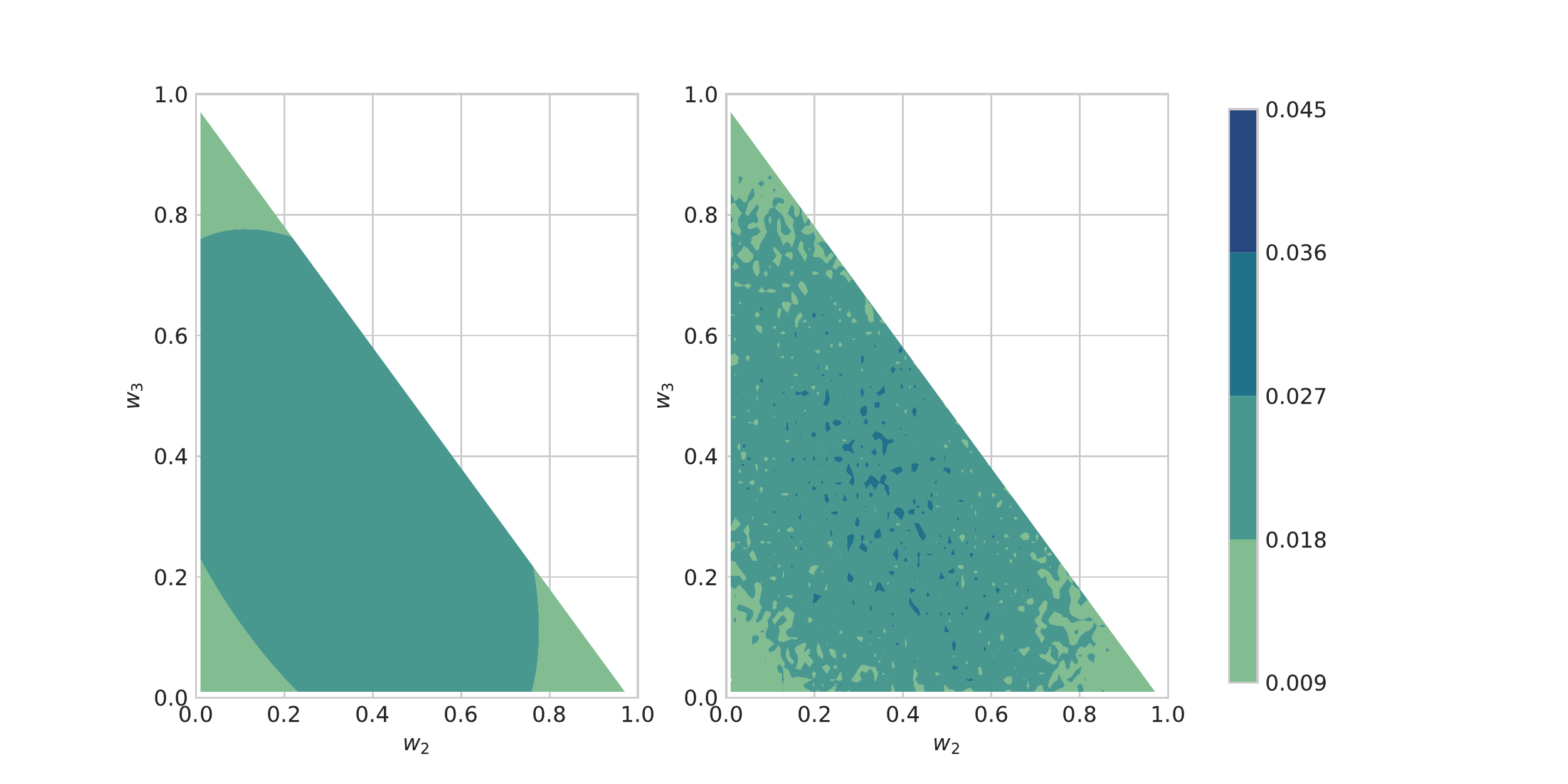}
                  \label{fig:sfig4_exp2}
                \end{subfigure}
                \begin{subfigure}{.333\textwidth}
                  \centering
                  \caption{\textbf{LOG (\ref{sym_log_model}, $\theta = 2.0$})}
                  \includegraphics[width=\linewidth]{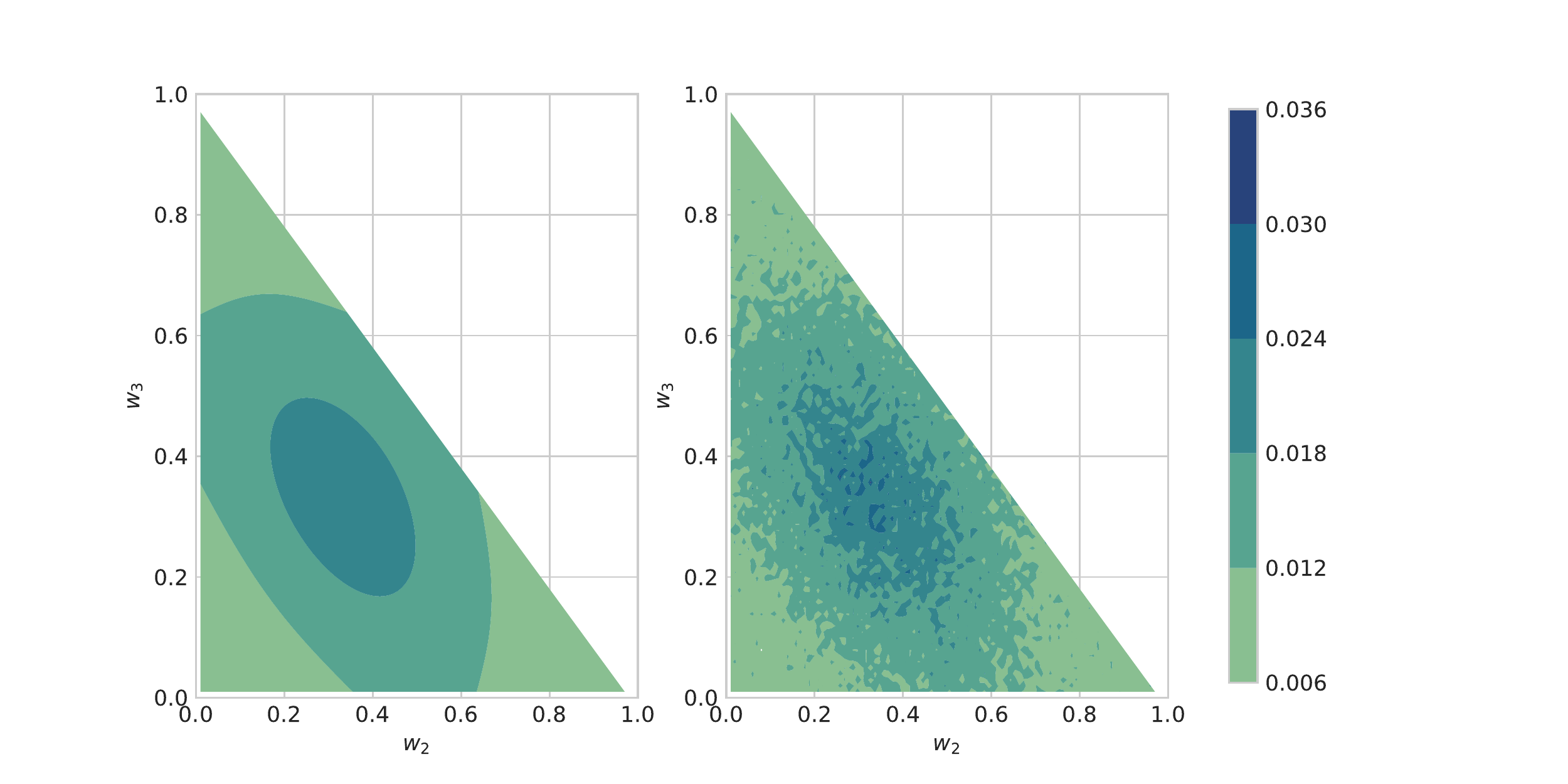}
                  \label{fig:sfig5_exp2}
                \end{subfigure}%
                \begin{subfigure}{.333\textwidth}
                  \centering
                  \caption{\textbf{ASL (\ref{asy_log_model}})}
                  \includegraphics[width=\linewidth]{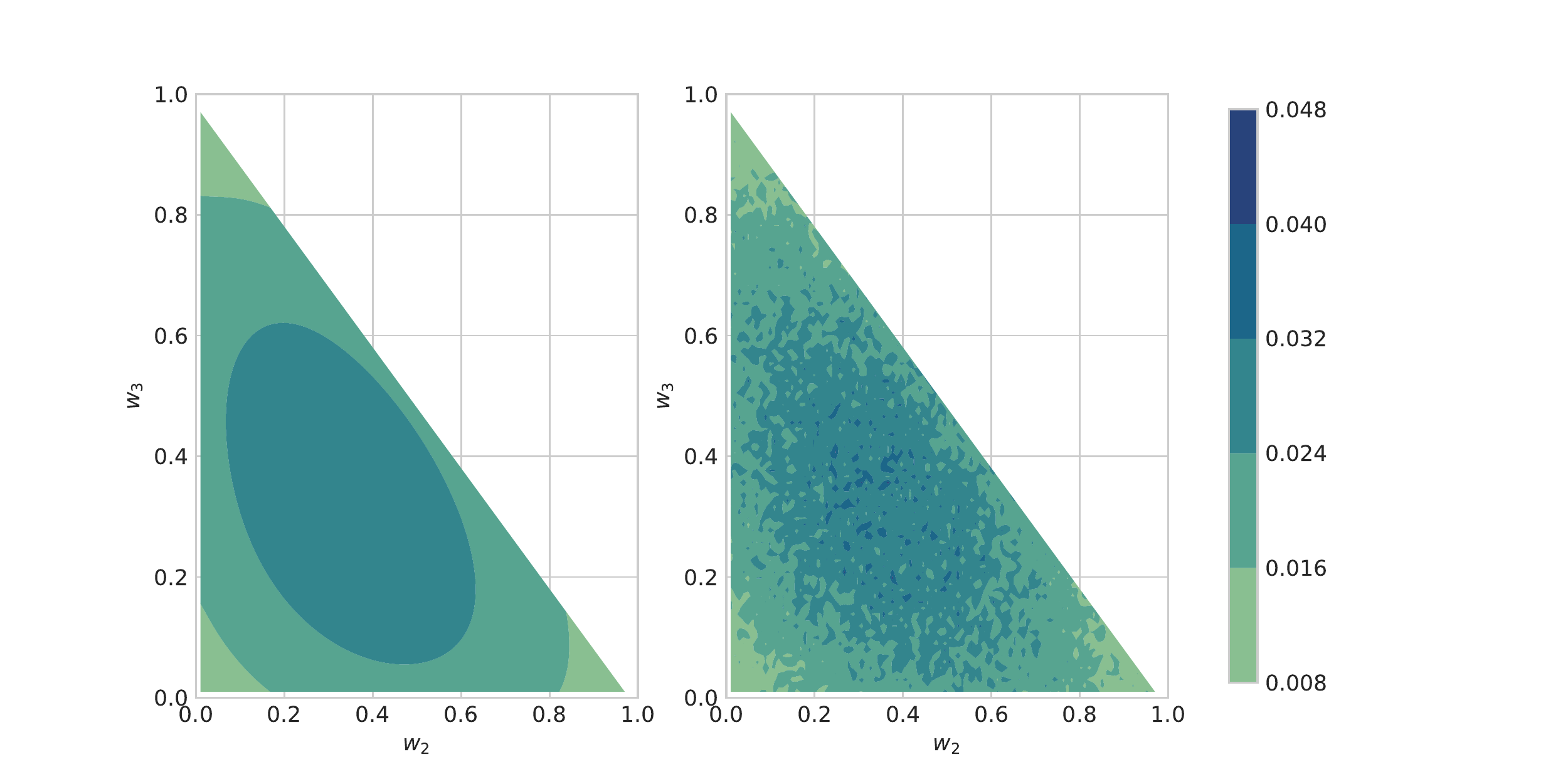}
                  \label{fig:sfig6_exp2}
                \end{subfigure}
                \hspace{\fill}
                \begin{subfigure}{.333\textwidth}
                  \centering
                  \includegraphics[width=\linewidth]{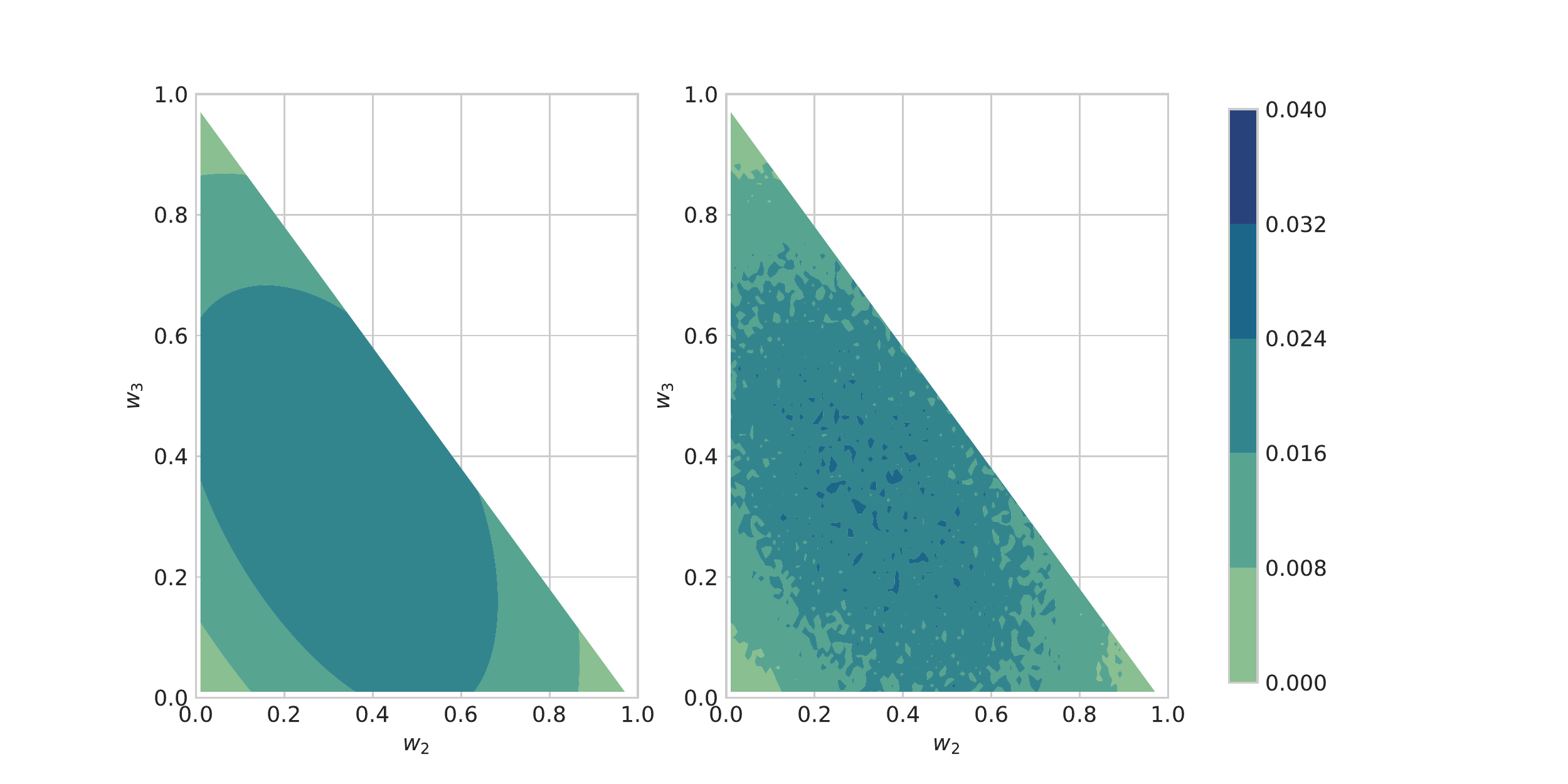}
                  \label{fig:sfig10_exp2}
                \end{subfigure}
                \begin{subfigure}{.333\textwidth}
                  \centering
                  \includegraphics[width=\linewidth]{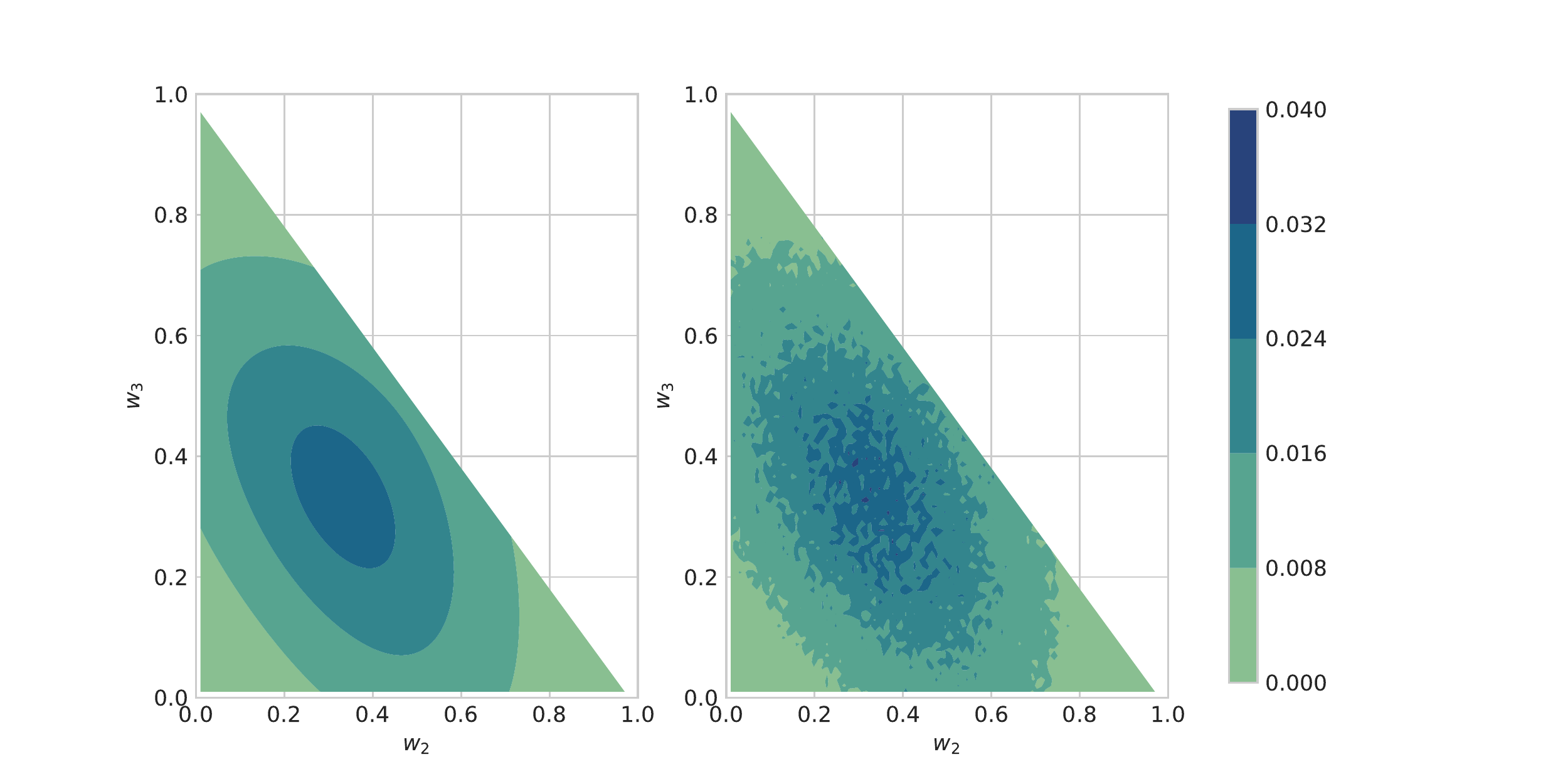}
                  \label{fig:sfig11_exp2}
                \end{subfigure}%
                \begin{subfigure}{.333\textwidth}
                  \centering
                  \includegraphics[width=\linewidth]{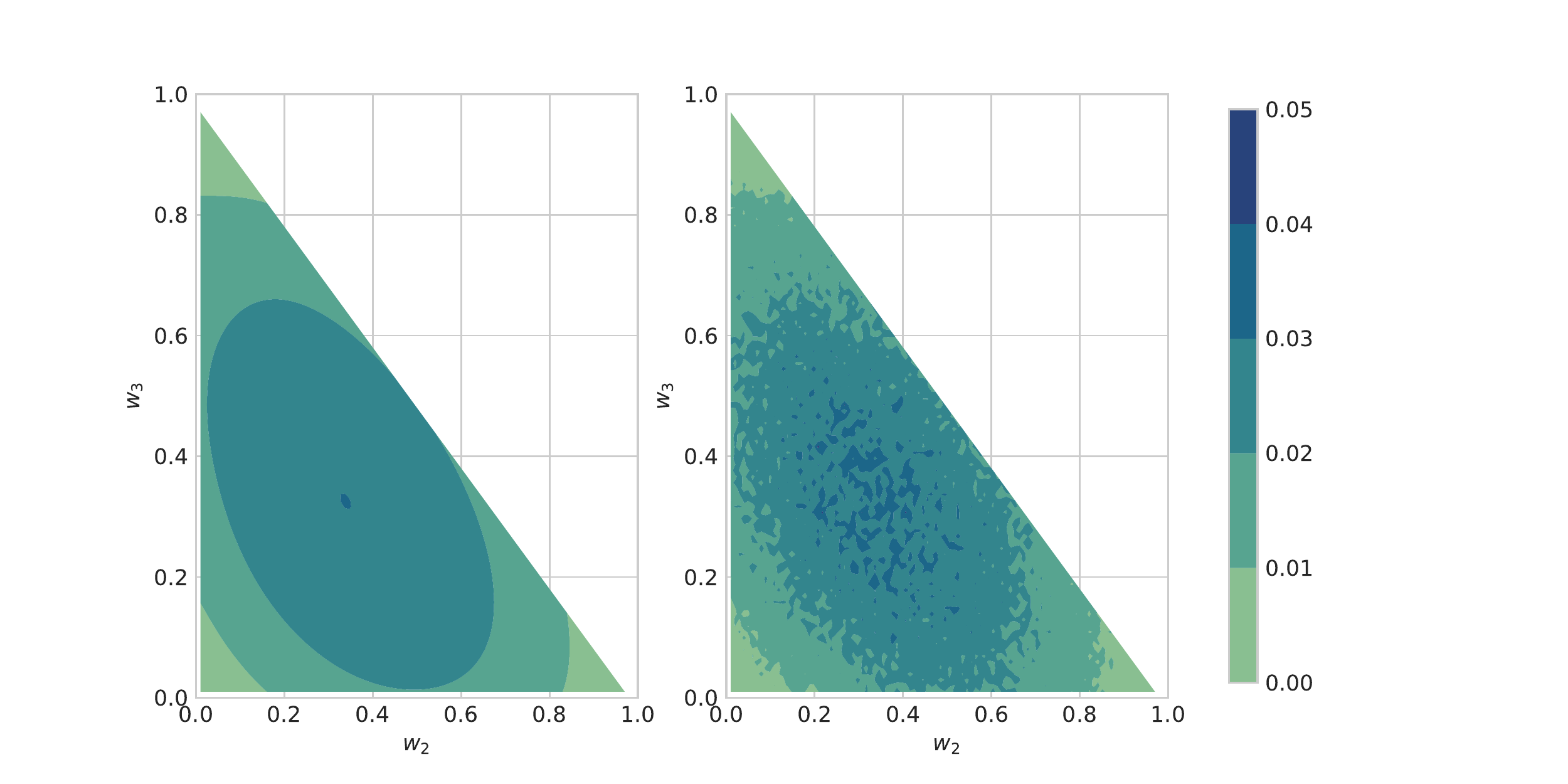}
                  \label{fig:sfig12_exp2}
                \end{subfigure}
                \hspace{\fill}
            \caption{Level sets of $\mathcal{E}_n^\mathcal{H}$ and $\mathcal{E}_n^\mathcal{H*}$, as a function of $\textbf{w}$, of the asymptotic variances of the estimators of the $\textbf{w}$-madogram. We present the level sets  corresponding to sufaces of Fig \ref{fig:missing_estimation_exp2}. On the left panel is represented the theoretical value given by Proposition \ref{Boulin} while on the right the empirical counterpart is given.}
            \label{fig:missing_estimation_exp3}
        \end{figure}

        In order to quantify errors in Figures \ref{fig:missing_estimation} and \ref{fig:missing_estimation_exp2}, in Table \ref{tab:mise} are displayed $\widehat{MISE}_n^{\mathcal{H}}$ and $\widehat{MISE}_n^{\mathcal{H}*}$ for the corresponding models in Experiments \textbf{E1} and \textbf{E2} to appreciate the proximity between the terms $\mathcal{E}_n^{\mathcal{H}}$ and $\mathcal{S}^{\mathcal{H}}$ (respectively for the corrected terms $\mathcal{E}_n^{\mathcal{H}*}$ and $\mathcal{S}^{\mathcal{H}*}$). As indicated by Figures \ref{fig:missing_estimation} and \ref{fig:missing_estimation_exp2}, errors in Table \ref{tab:mise} are close to zero.
        
        \begin{table*}[!htp] \centering
            \begin{tabular}{cccccccccc}\toprule
           & \multicolumn{6}{c}{\textbf{E1}} & \multicolumn{3}{c}{\textbf{E2}} \\
\cmidrule(lr){2-7}\cmidrule(lr){8-10}
            \textbf{MISE ($\times 10^{-5}$)} & \textbf{GAL} & \textbf{ANL} & \textbf{ASL} & \textbf{ASM} & \textbf{HR} & \textbf{tEV} & \textbf{IND} & \textbf{LOG} & \textbf{ASL} \\ \midrule
            \textbf{$\widehat{MISE}^{\mathcal{H}}_n$} & 2.49 & 8.10 & 2.43 & 1.85 & 1.89 & 1.93& 2.93 & 1.31 & 3.40  \\ \hdashline
            \textbf{$\widehat{MISE}^{\mathcal{H}*}_n$} & 2.77 & 7.02 & 2.04 & 1.94 & 1.96 & 1.93&  1.95 & 1.57 & 2.91 \\
            \bottomrule
            \end{tabular}
            \caption{Estimation of $MISE^{\mathcal{H}}$ and $MISE^{\mathcal{H}*}$ ($\times 10^{-5}$) defined in \eqref{eq:mise} for Experiment \textbf{E1} in the sixth first columns and \textbf{E2} in the last three columns.}
            \label{tab:mise}
        \end{table*}
        
        Figure \ref{fig:boxplot} illustrates the results of Experiment \textbf{E3} where we have drawn boxplots for Equation \eqref{rapport}. Not surprisingly, we observe that both the size of the boxplots and the median value  are increasing with $d$. However, this augmentation drops as the sample size $n$ increases and seems to appear as reasonable and able to handle the case of rather high dimensional data. A limitation (due to computation time issues) of this figure is that the number of points on the simplex is constant ($=300$) as a function of the dimension.
        
        \begin{figure}[!htp]
            \centering
            \includegraphics[scale = 0.75]{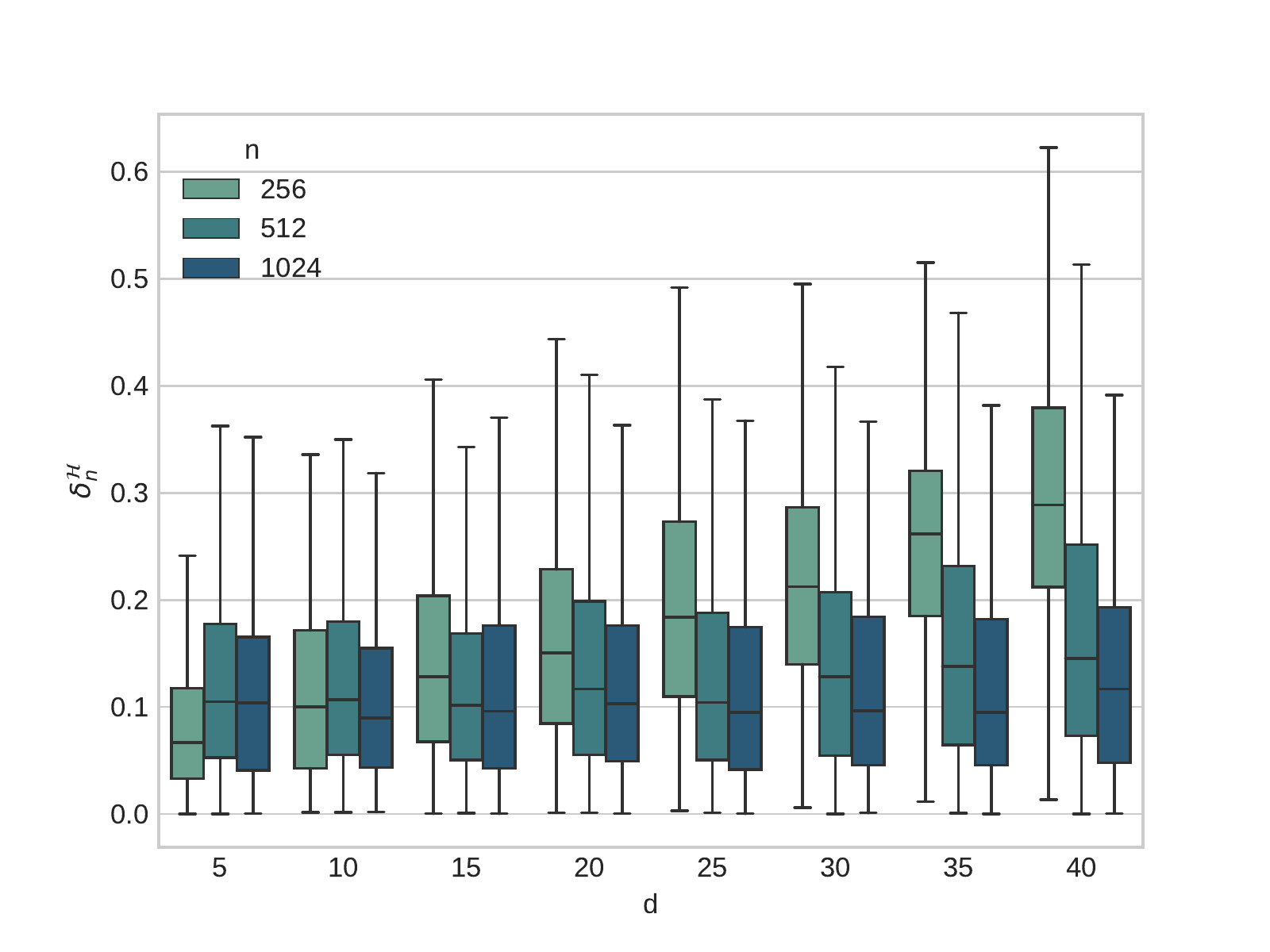}
            \caption{Boxplots for $\delta_n^{\mathcal{H}}$ for different values of $d$ and $n$.}
            \label{fig:boxplot}
        \end{figure}
        
        \section{Extremal dependence rainfall analysis via hybrid madogram}
        \label{sec:4}
        In climate studies, extreme events such as heavy precipitations  represent major challenge since damages from extreme weather events may have heavy consequences in both economic and human terms. Their spatial characteristics are of a prime interest and $\textbf{w}$-madogram and its estimator  studied in this paper  (see Equation \eqref{hybrid_lambda_fmado})   are able to capture those characteristics. A seminal application which bridges extreme value theory and geostatistics is the study of extreme rainfall since we expect spatial dependence among the recording weather stations. Precisely, we observe daily precipitation at station $j \in \{1,\dots,d\}$ over $n$ years. Concerning extreme events, one cannot use directly the observation for inference and we focus on block maxima. The block maxima approach is based on the observation of a sample of block maxima $\textbf{X}_i = (X_{i,1}, \dots, X_{i,d})$ where $X_{i,j}$ corresponds to the maximum at station $j \in \{1,\dots,d\}$ within the $i$th disjoint block of observation. A block could be either hourly, daily or annual for example. Consistent to our approach, we do not observe $\textbf{X}_i$ but an incomplete vector $\tilde{\textbf{X}}_i \in \bigotimes_{j=1}^d (\mathbb{R}_+ \cup \textup{NA})$. Our main goal is to estimate the extremal dependence between maxima of groups of station. This will be done for several clusters within which similar climate characteristics are envisaged leading to dependence among extremes.
        
        For each cluster, we compute the corrected hybrid madogram in Equation \eqref{hybrid_lambda_fmado}. This quantity is used to estimate the extremal coefficient (see for instance \cite{Smith2005MAXSTABLEPA}), using the relation between the Pickands and the madogram given in  Equation \eqref{w_mado_pickands}, defined by
        \begin{equation}
            \label{eq:ext_coeff}
           \theta = d\, A \left(\frac{1}{d}, \dots, \frac{1}{d}\right).
        \end{equation}
        
        This  satisfies the condition $1 \leq \theta \leq d$, where the lower and upper bounds represent the case of complete positive dependence and independence among the extremes, respectively. Since its upper bound depends on $d$, the extremal coefficient can, alas, only be used to compare clusters of the same size. In each cluster, the extremal coefficient in \eqref{eq:ext_coeff} is estimated by $\hat{\theta}_n = d \, \hat{A}_n^{\mathcal{H}*}$ where $\hat{A}_n^{\mathcal{H}*}$ is given in Definition \ref{def:pick_hat}.
        
        
        We illustrate the proposed methodology on rainfall data measured in millimimeter registred in 95 stations in Center Eastern Canada for a duration of 24 hours publicly available in the section \href{https://climate.weather.gc.ca/prods_servs/engineering_e.html}{engineering climate datasets} of the Government of Canada website. Annual maxima precipitations for a 24-hour duration are recorded from 1914 to 2017. The location of stations in Fig. \ref{fig:app} are given in the WGS84 coordinate space in order to have Euclidean distance between the stations and taking account of the geodesic geometry of the Earth. A specific characteristic of the considered rainfall data is the sparsity of the recorded data, \emph{i.e.} a lot of recordings are missing (see \cite{palaciosrodriguez:hal-03355026} for details). Four stations were removed of the analysis due to a tiny coverage of the observation period. As the measurements are maxima over a long period of time, it is reasonable to assume that they come from a multivariate extreme value distribution see Equation \eqref{evc}. The dataset we consider in this section and codes are available in the \href{https://github.com/Aleboul/missing/tree/main/fig_5}{github repository}.
        
        With the remaining 91 stations, we compare the extremal dependence between several groups of stations as it has been done by \cite{MARCON20171} (see Section 5) for France using a dataset with complete observations. We emphasize that the comparison of the extremal coefficient is solely relevant when clusters are of the same size. Thus, clusters were obtained by running the constrained $k$-means algorithm on the station coordinates (see for instance \cite{bradley2000constrained}) by forcing clusters of the same size : $d = 7$ or $d = 13$, \emph{i.e.} $13$ groups of $7$ stations and \emph{vice versa}. As overlapping data naturally decrease as the size of clusters increases, the case of cluster size $d = 13$ cannot be considered here. Among the 13 clusters of size $d = 7$, we only keep those having at least $10$ overlapping annual maxima within the cluster which results on $7$ remaining clusters depicted in Figure \ref{fig:sfig_clust}. The estimated coefficient range is between $3.88$, indicating strong dependence, and $5.07$, indicating medium dependence (see Figure \ref{fig:sfig_ext_coeff}).
        
        \begin{figure}[!htp]
                \begin{subfigure}{.5\textwidth}
                  \centering
                  \includegraphics[scale = 0.5]{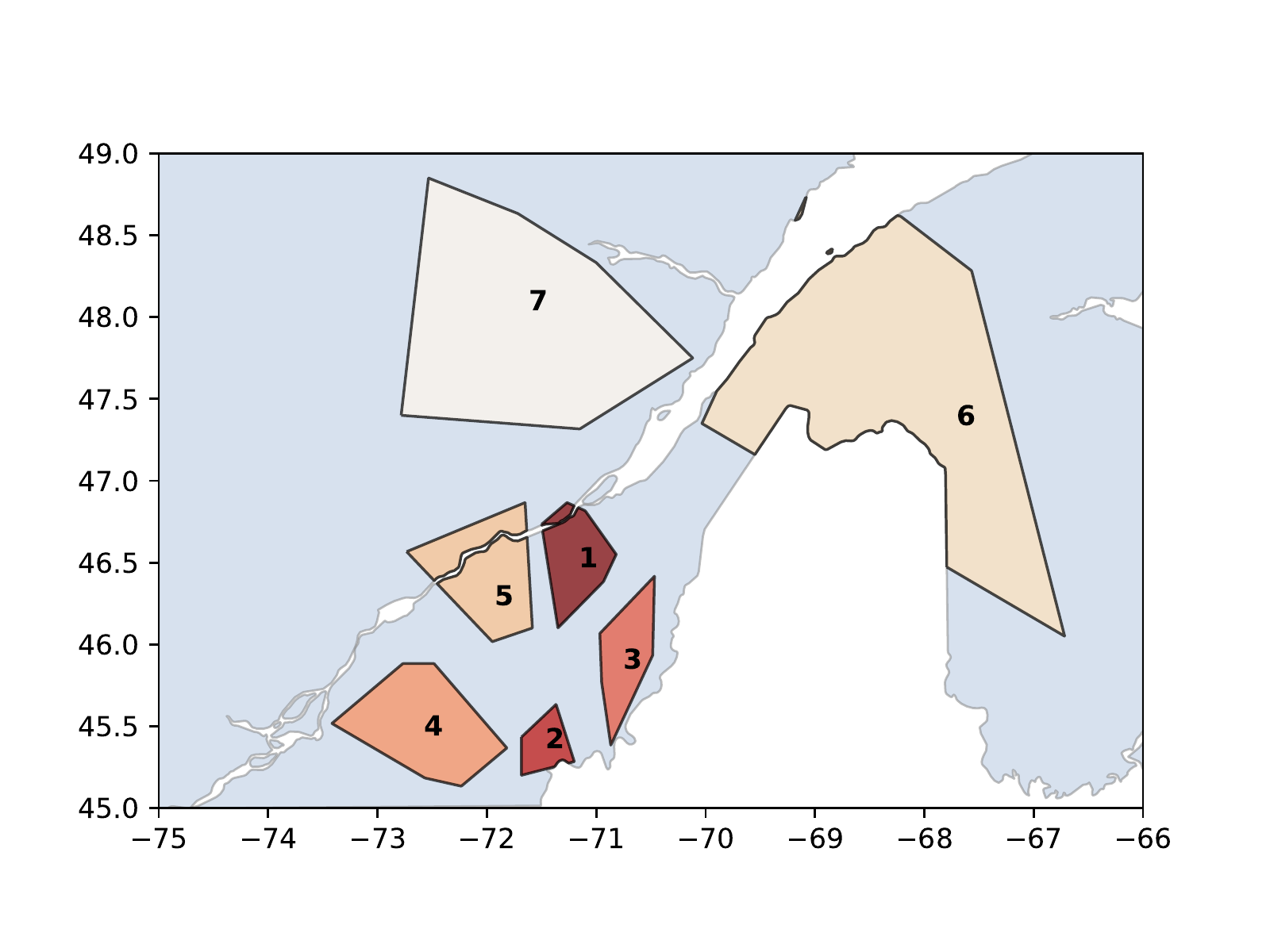}
                  \caption{Resulting clusters using constrained $k$-means}
                  \label{fig:sfig_clust}
                \end{subfigure}%
                \begin{subfigure}{.5\textwidth}
                  \centering
                  \includegraphics[scale = 0.5]{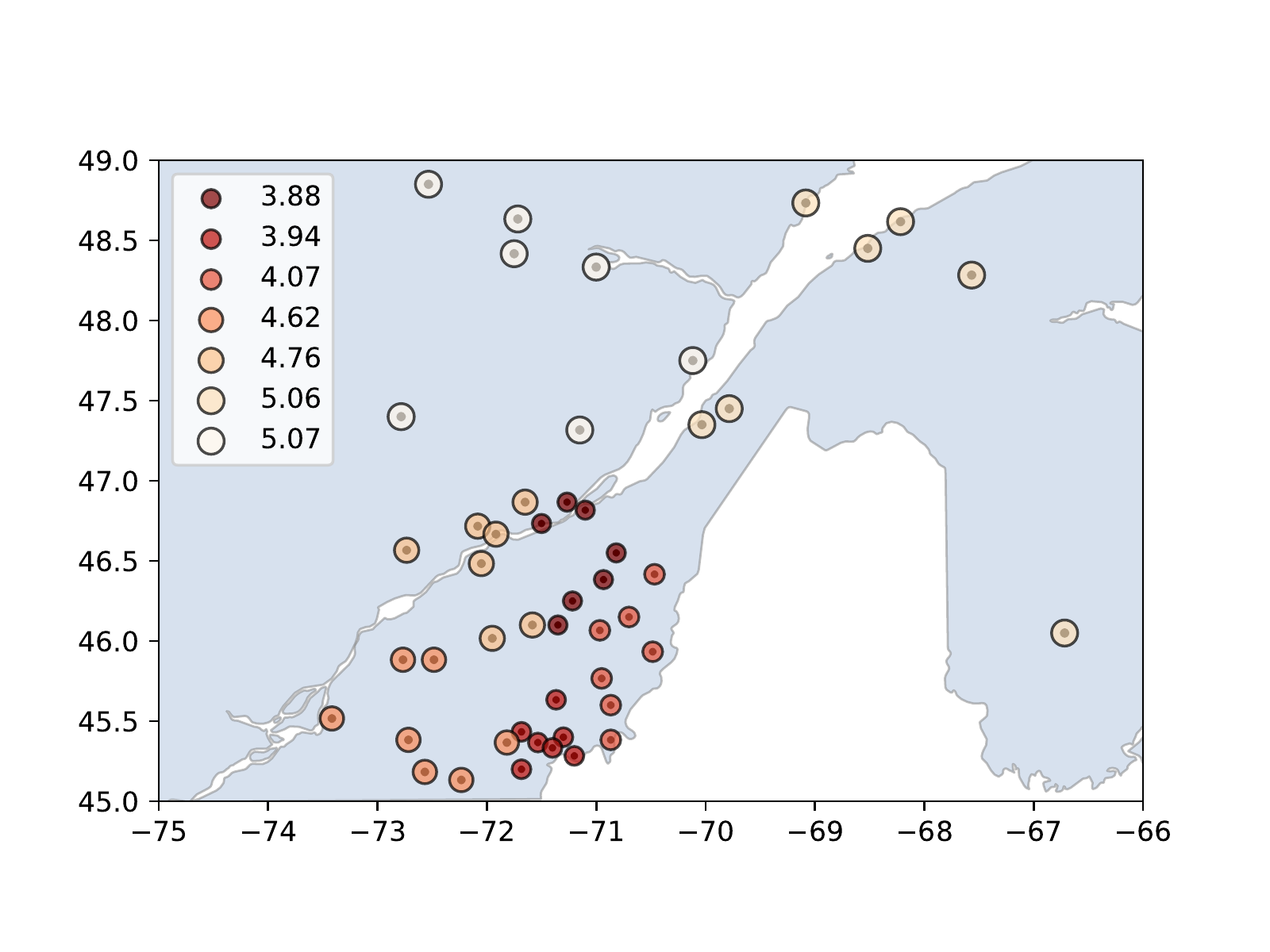}
                  \caption{Values of the extremal coefficient for each cluster}
                  \label{fig:sfig_ext_coeff}
                \end{subfigure}
            \caption{Analysis of Canadian annual rainfall maxima in the period 1914-2017. \textbf{(a)} Spatial representation of the $7$ selected clusters obtained via the constrained k-means algorithm. (\textbf{b}) Clusters of $49$ weather stations and their estimated extremal coefficients (with $d = 7$) obtained with the corrected version of the hybrid madogram.}
            \label{fig:app}
        \end{figure}
        Our estimations suggest an acute dependence among extremes in clusters $1$-$3$ in Figure \ref{fig:sfig_clust}. We can observe in Figure \ref{fig:sfig_ext_coeff} that extreme precipitations are more likely to be dependent in  the central coastal Atlantic region, \emph{a contrario}, one can notice a weak dependence among extreme values in the scattered clusters in the north of the region.
        
\section{Conclusions\label{sec:5}}
    A method based on madograms to estimate multivariate extremal dependencies with allowing missing data has been developed in this paper. Under the \textbf{MCAR} hypothesis, we studied the asymptotic behaviour for the proposed estimators. This approach is of interest to study spatio-temporal process ponctually observed as observations may not overlap. Moreover, we have derived closed expressions of their respective asymptotic variances for a fixed element in the simplex. Numerical results in a finite sample setting give  further evidences to our theoretical results and on performances of the proposed estimators of the madogram in the missing data setting. Finally, we applied our approach to the study of extremal dependencies of annual maxima of daily rainfall in Central Eastern Canada.
    
    As for future work, an interesting improvement could be to lower the \textbf{MCAR} assumption on the misssing data. Indeed, estimating nonparametrically the empirical copula process with missing data outside this framework is still unexplored. As a starting point, semiparametric inference for copula and copula based-regression allowing missing data under \textbf{M}issing \textbf{A}t \textbf{R}andom (\textbf{MAR}) mechanism have been studied by \cite{HAMORI201985} and \cite{HAMORI2020104654}.
    
    Another interesting direction could also be to build a dissimiliraty measure based on the bivariate $\textbf{w}$-madogram for clustering. This approach was already tackled by \cite{bernard:hal-03207469}, \cite{BADOR201517} and \cite{saunders} to partition respectively France, Europe and Australia with respect to extreme observations using the sole madogram. The idea here could be to use the infimum or the integral over $\textbf{w} \in (0,1)$ of the bivariate $\textbf{w}$-madogram as a dissimilarity measure and to show its strong consistency in the sense formulated by \cite{10.1214/aos/1176345339}. One limitation of our application is that clusters of same size is mandatory to compare the estimated extremal coefficient between clusters in Equation \eqref{eq:ext_coeff}. This feature stems from the bounds of the Pickands dependence function which depends on the dimension of the extremal random vector. Further investigations are thus needed to interpret extremal coefficient between clusters of different sizes, \emph{e.g.} to assess asymptotic independence between two extremal random vectors.

\section*{Acknowledgments}
This work has been supported by the project  ANR McLaren (ANR-20-CE23-0011).  This work has been partially supported  by  the French government, through the 3IA C\^{o}te d'Azur Investments in the Future project managed by the National
Research Agency (ANR) with the reference number ANR-19-P3IA-0002. This work was also supported by the french national programme LEFE/INSU.
 
\appendix
    \section{Proofs}
    \label{proof}
        \subsection{Proofs of  main results}
                For the rest of this section, we will write, for notational convenience, $n_i = \Pi_{j=1}^d I_{i,j}$ and $N = \sum_{i=1}^n n_i$. The following proof gives arguments used to establish the functional central limit theorem of our processes defined in Equation (\ref{processes}). Before going into details, we need an intermediary lemma to assert that the empirical cumulative distribution functions in case of missing data  verify Assumption \ref{Cond_3} and give covariance functions of the asymptotic processes $\alpha$ and $\beta_j$ with $j\in \{1,\dots,d\}$. This result comes down from \cite{segers2014hybrid} (see Example 3.5) where the result was proved for bivariate random variables but the higher dimension is directly obtained using same arguments.
                \begin{lemma}
                    \label{lemma_1}
                  Let $(\sqrt{n} (\hat{F}_n - F); \sqrt{n}(\hat{F}_{n,1} - F_1), \dots,\sqrt{n}(\hat{F}_{n,d} - F_d))$ with $\hat{F}_n$ and $\hat{F}_{n,j}$ for $j \in \{1,\dots,d\}$ as in  \eqref{F_X_G_Y_missing}. Then   Assumption \ref{Cond_3}  is satisfied with                   \begin{align*}
                        &\beta_j (u_j) = p_j^{-1} \mathbb{G}\left(\mathds{1}_{\{X_j \leq F_j^{\leftarrow}(u_j), I_j = 1\}} - u_j \mathds{1}_{\{I_j = 1\}} \right), \quad j \in \{1,\dots,d\},\\
                        &\alpha(\textbf{u}) = p^{-1} \mathbb{G}\left(\mathds{1}_{\{\mathbf{X} \leq \mathbf{F}^{\leftarrow}_d(\mathbf{u}),\emph{\textbf{I}}=\emph{\textbf{1}}\}} - C(\emph{\textbf{u}}) \mathds{1}_{\{\emph{\textbf{I}}=\emph{\textbf{1}}\}} \right),
                    \end{align*}
                    where $\mathbb{G}$ is a tight Gaussian process. Furthermore the covariance functions of the processes $\beta_j(u_j)$, $\alpha(\textbf{u})$,  for $(\emph{\textbf{u}},\emph{\textbf{v}},v_k) \in [0,1]^{2d+1}$, $j \in \{1,\dots,d\}$ and $j < k$, are given by 
                    \begin{align*}
                        & cov\left(\beta_j(u_j), \beta_j(v_j) \right) = p_j^{-1}\left( u_j \wedge v_j - u_j v_j \right), \\
                        & cov\left(\beta_{j}(u_j), \beta_k(v_k) \right) = \frac{p_{jk}}{p_j p_k} \left( C(\boldsymbol{1}_{j,k}(u_j,v_k)) - u_jv_k \right),\\
                        & cov\left(\alpha(\emph{\textbf{u}}), \alpha(\emph{\textbf{v}})\right) =  p^{-1} \left( C(\emph{\textbf{u}} \wedge \emph{\textbf{v}}) - C(\emph{\textbf{u}}) C(\emph{\textbf{u}}) \right), \\
                        & cov\left(\alpha(\emph{\textbf{u}}), \beta_j(v_j)\right) = p_j^{-1}\left( C(\emph{\textbf{u}}_j(u_j\wedge v_j)) - C(\emph{\textbf{u}}) v_j \right),
                    \end{align*}
                    where $\emph{\textbf{u}} \wedge \emph{\textbf{v}}$ denotes the vector of componentwise minima and $p_{jk} = \mathbb{P}(I_j = 1, I_k = 1)$. 
                \end{lemma}
                Proof of Lemma \ref{lemma_1} is postponed to \ref{proof_covariance_function_hybrid}.
                \label{proof_weak_conv_hybrid_mado}
                \begin{proof}[Proof of Theorem \ref{weak_conv_hybrid_mado}]
                    First, let us define the rank-corrected hybrid copula process suited with our estimator and its associated empirical copula process by
                    \begin{equation*}
            			\hat{C}_n^{\mathcal{R}}(\textbf{u}) = \frac{1}{\sum_{i=1}^n \Pi_{j=1}^d I_{i,j}} \sum_{i=1}^n \Pi_{j=1}^d \mathds{1}_{\left\{ \widetilde{U}_{i,j} \leq u_j \right\}} I_{i,j}, \quad \mathbb{C}_n^{\mathcal{R}} = \sqrt{n}\left(\hat{C}_n^{\mathcal{R}}- C\right).
            		\end{equation*}
            		One can show that
            		 \begin{equation*}
                    	\underset{\textbf{u} \in [0,1]^d}{\sup} \left| \hat{C}_n^\mathcal{H}(\textbf{u}) - C_n^{\mathcal{R}}(\textbf{u}) \right| \leq \frac{2d}{n \hat{p}_n},
                    \end{equation*}
                    with $\hat{p}_n = n^{-1} \sum_{i=1}^n \Pi_{j=1}^d I_{i,j}$. Note that $\hat{p}_n$ converges in probability to $p \in ]0,1]$ which implies that the difference between $\hat{C}_n^\mathcal{H}$ and $\hat{C}_n^{\mathcal{R}}$ is asymptotically negligible.
                    Details for the proof are given solely for the estimator $\hat{\nu}_n^\mathcal{H*}$ as the weak convergence for $\hat{\nu}_n^\mathcal{H}$ is obtained similarly via an adequate continuous transformation of $\hat{\nu}_n^\mathcal{H}$ with $\hat{C}_n^{\mathcal{R}}$. Using that $\mathbb{E}[F_j(X_j)^\alpha] = (1+\alpha)^{-1}$ for $\alpha \neq -1$, we can write $\nu(\textbf{w})$ as :
                    \begin{align*}
                        \nu(\textbf{w}) =& \mathbb{E}\left[\bigvee_{j=1}^d \left\{ F_j(X_j) \right\}^{1/w_j} - \frac{1}{d} \sum_{j=1}^d \left\{F_j(X_j)\right\}^{1/w_j} \right] +
                        \sum_{j=1}^d \frac{\lambda_j(\textbf{w})(d-1)}{d} \left( \frac{w_j}{1+w_j} - \mathbb{E}\left[ \left\{F_j(X_j)\right\}^{1/w_j} \right]\right) \\
                        =& \mathbb{E}\left[\bigvee_{j=1}^d \left\{ F_j(X_j) \right\}^{1/w_j}\right] - \frac{1}{d}\sum_{j=1}^d (1+\lambda_j(\textbf{w})(d-1)) \mathbb{E}\left[ \left\{F_j(X_j)\right\}^{1/w_j} \right] + a(\textbf{w}),
                    \end{align*}
                    with $a(\textbf{w}) = (d-1)d^{-1} \sum_{j=1}^d \lambda_j(\textbf{w})w_j / (1+w_j)$. Let us note by $g_\textbf{w}$ the function defined as
                    \begin{equation*}
                        g_{\textbf{w}} : [0,1]^d \rightarrow [0,1], \quad \textbf{u} \mapsto \bigvee_{j=1}^d u_j^{1/w_j} - \frac{1}{d}\sum_{j=1}^d(1+\lambda_j(\textbf{w})(d-1)) u_j^{1/w_j}.
                    \end{equation*}
                    One can write our estimator of the $\textbf{w}$-madogram and the theoretical $\textbf{w}$-madogram in missing data framework as an integral with respect to the rank-corrected hybrid copula estimator and the copula function, respectively. We thus have:
                    \begin{align*}
                        \hat{\nu}_n^\mathcal{H*}(\textbf{w}) &= \frac{1}{N} \sum_{i=1}^n g_{\textbf{w}}\left(\widetilde{U}_{i,1}, \dots, \widetilde{U}_{i,d})\right) \Pi_{j=1}^d I_{i,j} + a(\textbf{w}) = \int_{[0,1]^d} g_{\textbf{w}}\left(\textbf{u}\right) d\hat{C}_n^{\mathcal{R}}(\textbf{u}) + a(\textbf{w}), \\
                        \nu(\textbf{w}) &= \int_{[0,1]^d} g_{\textbf{w}}\left(\textbf{u}\right) dC(\textbf{u}) + a(\textbf{w}).
                    \end{align*}
                    We obtain, proceeding as in Theorem 2.4 of \cite{MARCON20171} :
                    \begin{align*}
                        \sqrt{n}\left(\hat{\nu}_n^\mathcal{H*}(\textbf{w}) - \nu(\textbf{w}) \right) =& \frac{1}{d}\sum_{j=1}^d \left(1+\lambda_j(\textbf{w})(d-1)\right) \int_{[0,1]} \mathbb{C}_n^\mathcal{R}(\textbf{1}_j(x^{w_j})) dx - \int_{[0,1]} \mathbb{C}_n^\mathcal{R}\left(x^{w_1}, \dots, x^{w_d}\right) dx,
                    \end{align*}
                    where $\textbf{1}_j(u)$ denotes the vector composed out of $1$ except for the jth component where $u$ does stand and with $\mathbb{C}_n^\mathcal{H}$ in \eqref{hybrid_copula}. Consider the function $\phi : \ell^\infty([0,1]^d) \rightarrow \ell^\infty(\Delta^{d-1}),\,  f \mapsto \phi(f)$, defined by
                    \begin{equation*}
                        (\phi)(f)(\textbf{w}) = \frac{1}{d} \sum_{j=1}^d (1+\lambda_j(\textbf{w})(d-1)) \int_{[0,1]}f(\textbf{1}_j(x^{w_j}))dx - \int_{[0,1]} f(x^{w_1}, \dots, x^{w_d})dx.
                    \end{equation*}
                    This function is linear and bounded thus continuous. The continous mapping theorem (see,  \emph{e.g.}, Theorem 1.3.6 of \cite{vaartwellner96book}) implies, as $n \rightarrow \infty$
                    \begin{equation*}
                        \sqrt{n}(\hat{\nu}_n^\mathcal{H*} - \nu) = \phi(\mathbb{C}_n^\mathcal{R}) \rightsquigarrow \phi(S_C),
                    \end{equation*}
                    in $\ell^\infty(\Delta^{d-1})$. Recall that $S_C$ is the asymptotic process where $\mathbb{C}_n^\mathcal{H}$ does converge in the sense of the weak convergence in $\ell^\infty(\Delta^{d-1})$ and is defined by $S_C(\textbf{u}) = \alpha(\textbf{u}) - \sum_{j=1}^d \beta_j(u_j) \dot{C}_j (\textbf{u})$ with $\textbf{u} \in [0,1]^d$ and $\alpha$ and $\beta_j$ are processes defined in Lemma \ref{lemma_1}. We note that $S_C(\textbf{1}_j(x^{w_j})) = \alpha(\textbf{1}_j(x^{w_j})) - \beta_j(u_j)$ and we obtain our statement.
                \end{proof}
                The asymptotic normality of our estimators   directly comes down from being a linear transformation of a tight Gaussian process for $\textbf{w} \in \Delta^{d-1}$. The proof below uses technical arguments to exhibit the closed expressions of the asymptotic variances of the Gaussian  limit distributions of our estimators   in Equation (\ref{hybrid_lambda_fmado}) and (\ref{corrected_lambda_FMado_hybrid}). Furthermore, this proof strengthen our choice of the definition of the corrected estimator. Indeed, the chosen form of the corrected estimator makes computations more tractable as we only have to compute terms for the hybrid estimator and to multiply those by different factors. Two tools make the computation feasible. The first one is the form exhibited by Equation \eqref{pick_tail} which transforms a double integral with respect to the trajectory of the copula function as the double integral of a power function. When this trick is not possible, again the expression of the extreme value copula with respect to the Pickands dependence function is of main interest. Indeed, with some substitutions, we are able to express the double integrals as the integral with respect to the Pickands dependence function using the following equality :
                \begin{equation*}
                     -\int_{[0,1]} w^\alpha \ln(w) \,dw =  \frac{1}{(\alpha +1)^2},
                \end{equation*}
                where $\alpha \neq -1$. 
                \begin{proof}[\textbf{Proof of Proposition \ref{Boulin}}]
                    \label{proof_Boulin}
                    Recall that  $\emph{\textbf{p}} = (p_1, \dots, p_d, p)$.                     By definition the asymptotic variance $\mathcal{S}^\mathcal{H}(\textbf{p}, \textbf{w})$ for a fixed $\textbf{w} \in \Delta^{d-1}$  is given   by
                    \begin{equation*}
                        \mathcal{S}^\mathcal{H}(\textbf{p}, \textbf{w}) := Var\left(\frac{1}{d} \sum_{j=1}^d \int_{[0,1]} \alpha(\textbf{1}_j(x^{w_j})) - \beta_j(x^{w_j})dx- \int_{[0,1]} S_C(x^{w_1}, \dots, x^{w_d}) dx \right).
                    \end{equation*}
                    Using properties of the variance operator, we thus obtain 
                    \begin{align*}
                        \mathcal{S}^\mathcal{H}(\textbf{p}, \textbf{w}) &=\frac{1}{d^2} \sum_{j=1}^d Var\left( \int_{[0,1]} \alpha(\textbf{1}_j(x^{w_j})) - \beta_j(x^{w_j})dx\right) + Var\left( \int_{[0,1]} S_C(x^{w_1}, \dots, x^{w_d}) dx\right) \\ &+\frac{2}{d^2}\sum_{j <k}cov\left(\int_{[0,1]} \alpha(\textbf{1}_j(x^{w_j})) - \beta_j(x^{w_j})dx,\int_{[0,1]} \alpha(\textbf{1}_k(x^{w_k})) - \beta_k(x^{w_k})dx \right) \\&- 
                        \frac{2}{d} \sum_{j=1}^d cov \left(\int_{[0,1]} \alpha(\textbf{1}_j(x^{w_j})) - \beta_j(x^{w_j})dx, \int_{[0,1]} \alpha(x^{w_1}, \dots, x^{w_d})dx\right) \\
                        &+\frac{2}{d} \sum_{j=1}^d \sum_{k=1}^d cov \left(\int_{[0,1]} \alpha(\textbf{1}_j(x^{w_j})) - \beta_j(x^{w_j})dx, \int_{[0,1]} \beta_k(x^{w_k}) \dot{C}_k(x^{w_1}, \dots, x^{w_d})dx\right) .
                    \end{align*}
                    By definition of the covariance functions of $\alpha$ , $\beta_j$ with $j \in \{1,\dots,d\}$ given in Lemma \ref{lemma_1}, we have for the variance terms
                    \begin{align*}
                        \label{sigma_j}
                        Var\left( \int_{[0,1]} \alpha(\textbf{1}_j(x^{w_j})) - \beta_j(x^{w_j})dx\right) &= \left(p^{-1} - p_j^{-1} \right) \sigma_j^2(\textbf{w}), \\
                        Var\left( \int_{[0,1]} S_C(x^{w_1}, \dots, x^{w_d}) dx\right) &= \sigma_{d+1}^2(\textbf{p}, \textbf{w}).
                    \end{align*}
                    We obtain similarly for the covariance terms 
                    \begin{align*}
                        cov\left(\int_{[0,1]} \alpha(\textbf{1}_j(x^{w_j})) - \beta_j(x^{w_j})dx,\int_{[0,1]} \alpha(\textbf{1}_k(x^{w_k})) - \beta_k(x^{w_k})dx \right) &=\left(p^{-1}-p_j^{-1}-p_k^{-1}+\frac{p_{jk}}{p_jp_k}\right) \sigma_{jk}(\textbf{w}), \\
                        cov \left(\int_{[0,1]} \alpha(\textbf{1}_j(x^{w_j})) - \beta_j(x^{w_j})dx, \int_{[0,1]} \alpha(x^{w_1}, \dots, x^{w_d})dx\right) &= \left(p^{-1} - p_j^{-1} \right) \sigma_{j}^{(1)}(\textbf{w}),\\
                        cov \left(\int_{[0,1]} \alpha(\textbf{1}_j(x^{w_j})) - \beta_j(x^{w_j})dx, \int_{[0,1]} \beta_k(x^{w_k}) \dot{C}_k(x^{w_1}, \dots, x^{w_d})dx\right) &=
                        \left(p_k^{-1} - \frac{p_{jk}}{p_jp_k} \right)\sigma_{jk}^{(2)}(\textbf{w}).
                    \end{align*}
                    We first show in details the closed form for $\sigma^2_{d+1}$, the other forms are given without explanations as the technical tools used are those used for $\sigma^2_{d+1}$. Proceding as before, we decompose this quantity as a linear combination of the variance (the squared term $\gamma_1^2$ and $\gamma_j^2$ for $j \in \{1,\dots,d\}$) and the covariance terms ($\gamma_{1j}$ and $\tau_{jk}$) with the probabilities of missing. The explicit formula of these quantities will be defined below. We set
                    \begin{equation}
                        \label{Var_N_C}
                        \sigma_{d+1}^2(\textbf{p}, \textbf{w}) = p^{-1} \gamma_1^2(\textbf{w}) + \sum_{j=1}^d p_j^{-1} \gamma_j^2(\textbf{w})-2\sum_{j=1}^d p_j^{-1} \gamma_{1j}(\textbf{w}) + 2 \sum_{j < k} \frac{p_{jk}}{p_jp_k} \tau_{jk}(\textbf{w}).
                    \end{equation}
                    Let us   exhibit a useful form of the partial derivatives of the extreme value copula. We have $\forall j \in \{1,\dots,d\}$ :
                    \begin{equation*}
                        \dot{C}_j(\textbf{u}) = \frac{C(\textbf{u})}{u_j} \dot{\ell}_j(-\ln(u_1) , \dots, -\ln(u_d)).
                    \end{equation*}
                    Furthermore, as $\ell(x_1, \dots, x_d)$ is homogeneous of degree 1, the partial derivative $\dot{\ell}_j (x_1,\dots,x_d)$ is homogeneous of degree 0 for $j\in \{1,\dots,d\}$. We thus obtain a suitable form of the partial derivatives of the extreme value copula for $u \in ]0,1[$ and $\textbf{w} \in \Delta^{d-1}$ :
                    \begin{align*}
                        \dot{C}_j(u^{w_1}, \dots, u^{w_d}) &= \frac{u^{A(\textbf{w})}}{u^{w_j}} \dot{\ell}_j(-w_1\ln(u) , \dots, -w_d\ln(u)) = \frac{u^{A(\textbf{w})}}{u^{w_j}} \dot{\ell}_j(-w_1 , \dots, -w_d) = \frac{u^{A(\textbf{w})}}{u^{w_j}} \mu_j(\textbf{w}),
                    \end{align*}
                    where $\mu_j(\textbf{w}) \triangleq \dot{\ell}_j(-w_1, \dots, -w_d)$. Now, using linearity of the integral and the definition of the covariance function of $\alpha$, we obtain 
                    \begin{align*}
                        p^{-1}\gamma_1^2(\textbf{w}) &\triangleq \mathbb{E}\left[\int_{[0,1]} \alpha(u^{w_1}, \dots, u^{w_d}) du \int_{[0,1]} \alpha(v^{w_1}, \dots, v^{w_d}) dv\right] \\
                                                     & = \frac{2}{p}\int_{[0,1]}\int_{[0,v]}u^{A(\textbf{w})}(1-v^{A(\textbf{w})})duv.
                    \end{align*}
                    Let us compute
                    \begin{align*}
                        \gamma_1^2(\textbf{w}) = 2\int_{[0,1]}\int_{[0,v]}u^{A(\textbf{w})}(1-v^{A(\textbf{w})})duv = \frac{1}{(1+A(\textbf{w}))^2}\frac{A(\textbf{w})}{2+A(\textbf{w})}.
                    \end{align*}
                    The quantity $\gamma_j^2(\textbf{w})$ is defined by the following
                    \begin{align*}
                        p_j^{-1}\gamma_j^2(\textbf{w}) &\triangleq \mathbb{E}\bigg[ \int_{[0,1]} \beta_j(u^{w_j}) \dot{C}_j(u^{w_1}, \dots, u^{w_d})du\int_{[0,1]} \beta_j(u^{w_j}) \dot{C}_j(v^{w_1}, \dots, v^{w_d})dv\bigg] \\
                        &=  \frac{2}{p_j} \int_{[0,1]} \int_{[0,v]}u^{w_j}(1-v^{w_j})\mu_j(\textbf{w}) \mu_j(\textbf{w}) u^{A(\textbf{w})-w_j} v^{A(\textbf{w})-w_j}duv.
                    \end{align*}
                    It is clear that
                    \begin{align*}
                        \gamma_j^2(\textbf{w}) &= 2 \int_{[0,1]} \int_{[0,v]}u^{w_j}(1-v^{w_j})\mu_j(\textbf{w}) \mu_j(\textbf{w}) u^{A(\textbf{w})-w_j} v^{A(\textbf{w})-w_j}duv = \left(\frac{\mu_j(\textbf{w})}{1+A(\textbf{w})}\right)^2 \frac{w_j}{2A(\textbf{w})+1+1-w_j}.
                    \end{align*}
                    We now deal with cross product terms, the first we define is
                    \begin{align*}
                        p_j^{-1}\gamma_{1j}(\textbf{w}) &\triangleq \mathbb{E}\bigg[\int_{[0,1]}\alpha(u^{w_1}, \dots, u^{w_d})du \int_{[0,1]} \beta_j(v^{w_j}) \dot{C}_j(v^{w_1}, \dots, v^{w_d})dv \bigg] \\
                        &= p_j^{-1}\int_{[0,1]^2} \left(C(u^{w_1}, \dots, (u\wedge v)^{w_j}, \dots, u^{w_d}) - u^{A(\textbf{w})}v^{w_j} \right)\dot{C}_j(v^{w_1}, \dots, v^{w_d}) duv.
                    \end{align*}
                    Under the rectangle $[0,1]\times [0,v]$, we have
                    \begin{align*}
                        \gamma_{1j}(\textbf{w}) &=\int_{[0,1]\times [0,v]} \left(C(u^{w_1}, \dots, u^{w_j}, \dots, u^{w_d}) - u^{A(\textbf{w})}v^{w_j} \right)\dot{C}_j(v^{w_1}, \dots, v^{w_d}) duv \\
                        &= \int_{[0,1]\times [0,v]} u^{A(\textbf{w})}(1-v^{w_j})v^{A(\textbf{w})-w_j} \mu_j(\textbf{w})duv
                        = \frac{\mu_j(\textbf{w})}{2(1+A(\textbf{w}))^2} \frac{w_j}{2A(\textbf{w})+1+(1-w_j)}.
                    \end{align*}
                    Under the rectangle $[0,1]\times [0,u]$, we have for the right term
                    \begin{align*}
                        \int_{[0,1]\times [0,u]}u^{A(\textbf{w})}v^{w_j}v^{A(\textbf{w})-w_j}\mu_j(\textbf{w})dvu = \frac{\mu_j(\textbf{w})}{2(1+A(\textbf{w}))^2}. 
                    \end{align*}
                    For the left term, by definition, we have
                    \begin{equation*}
                        \int_{[0,1]\times [0,u]} C(u^{w_1}, \dots, v^{w_j}, \dots, u^{w_d})\dot{C}_j(v^{w_1}, \dots, v^{w_d}) dvu.
                    \end{equation*}
                    Let us consider the substitution $x = v^{w_j}$ and $y = u^{1-w_j}$, we obtain
                    \begin{align*}
                        \frac{1}{w_j(1-w_j)} \int_{[0,1]} \int_{[0, y^{w_j/(1-w_j)}]}C\left(y^{w_1/(1-w_j)}, \dots, x,\dots, y^{w_d/(1-w_j)}\right) \times \\ \dot{C}_j\left(x^{w_1/w_j}, \dots, x^{w_d/w_j}\right) x^{(1-w_j)/w_j} y^{w_j/(1-w_j)}dxy.
                    \end{align*}
                    Let us compute the quantity
                    \begin{align*}
                        \dot{C}_j(x^{w_1/w_j}, \dots, x^{w_d/w_j}) = \frac{C(x^{w_1/w_j}, \dots, x^{w_d/w_j})}{x} \mu_j(\textbf{w}).
                    \end{align*}
                    Using Equation (\ref{evc}), we have
                    \begin{align*}
                        C(x^{w_1/w_j}, \dots, x^{w_d/w_j}) &= \exp\left(-\ell\left(-\frac{\ln(x)}{w_j} w_1, \dots, \frac{\ln(x)}{w_j} w_d \right)\right) \\ &= \exp\left(-\frac{\ln(x)}{w_j} \ell\left(-w_1, \dots, -w_d\right)\right) = x^{A(\textbf{w})/w_j} = x^{A_j(\textbf{w})},
                    \end{align*}
                    where we use the homogeneity of order one of $\ell$ and that $-\ell(-w_1, \dots,-w_d) = A(\textbf{w})$ as stated by the identity of Equation (\ref{pick_tail}) and that $\textbf{w} \in \Delta^{d-1}$. Now, consider the substitution $x = w^{1-s}$ and $y = w^s$, the jacobian of this transformation is given by $-\ln(w)$, we have
                    \begin{align*}
                        -\frac{\mu_j(\textbf{w})}{w_j(1-w_j)} \int_{[0,1]} \int_{[0, 1-w_j]}C\left(w^{sw_1/(1-w_j)}, \dots, w^{1-s},\dots, w^{sw_d/(1-w_j)}\right) w^{(1-s)\left[A_j(\textbf{w})+\frac{1-w_j}{w_j}-1\right]+ s \frac{w_j}{1-w_j}}\ln(w)dsw,
                    \end{align*}
                    where we note by $A_j(\textbf{w}) := A(\textbf{w}) / w_j$ with $j \in \{1,\dots,d\}$. We now compute the quantity
                    \begin{equation*}
                        C\left(w^{sw_1/(1-w_j)}, \dots, w^{1-s},\dots, w^{sw_d/(1-w_j)}\right).
                    \end{equation*}
                    Using the same techniques as above, we have
                    \begin{align*}
                        C\left(w^{sw_1/(1-w_j)}, \dots, w^{1-s},\dots, w^{sw_d/(1-w_j)}\right) &= \exp\left( - \ell\left(-\frac{sw_1}{1-w_j}\ln(w), \dots, -(1-s)\ln(w),\dots, - \frac{sw_d}{1-w_j}\ln(w) \right) \right) \\
                        &= \exp\left( -\ln(w) \ell\left( -\frac{sw_1}{1-w_j}, \dots , -(1-s), \dots, -\frac{sw_d}{1-w_j}\right) \right).
                    \end{align*}
                    Now, using that $\textbf{w} \in \Delta^{d-1}$, remark that $s\sum_{i \neq j} w_i / (1-w_j) = s$, we have, using Equation (\ref{pick_tail})
                    \begin{equation*}
                        -\ell\left( -\frac{sw_1}{1-w_j}, \dots , -(1-s), \dots, -\frac{sw_d}{1-w_j}\right) = A\left(\textbf{z}_j(1-s)\right),
                    \end{equation*}
                    where $\textbf{z} = (sw_1/(1-w_j), \dots, sw_d/(1-w_j))$. So we have
                    \begin{align*}
                        \gamma_{1j}(\textbf{w}) &=  -\frac{\mu_j(\textbf{w})}{w_j(1-w_j)} \int_{[0, 1-w_j]} \int_{[0,1]}  w^{A\left(\textbf{z}_j(1-s)\right)+(1-s)\left(A_j(\textbf{w})+\frac{1-w_j}{w_j}-1\right)+ s \frac{w_j}{1-w_j}}\ln(w)dws \\
                        &= \frac{\mu_j(\textbf{w})}{w_j(1-w_j)} \int_{[0,1-w_j]} \bigg[A\left(\textbf{z}_j(1-s)\right)+(1-s)\left(A_j(\textbf{w})+\frac{1-w_j}{w_j}-1\right)+ s \frac{w_j}{1-w_j}+1\bigg]^{-2}ds.
                    \end{align*}
                    No further simplifications can be obtained. For $j < k$, let us define the quantity $\tau_{jk}$ such as
                    \begin{align}
                        \label{cov_ij}
                        \frac{p_{jk}}{p_jp_k}\tau_{jk}(\textbf{w}) &\triangleq \mathbb{E}\bigg[ \int_{[0,1]} \beta_j(u^{w_j}) \dot{C}_j(u^{w_1}, \dots, u^{w_d})du \int_{[0,1]}\beta_k(v^{w_k}) \dot{C}_k(v^{w_1}, \dots, v^{w_d})dv\bigg].
                    \end{align}
                    Again, we have
                    \begin{align*}
                        \tau_{jk}(\textbf{w})= \int_{[0,1]^2} \left(C(\textbf{1}_{jk}(u^{w_j}, v^{w_j})) - u^{w_j}v^{w_j}\right) \dot{C}_j(u^{w_1}, \dots, u^{w_d})\dot{C}_k(v^{w_1}, \dots, v^{w_d})duv.
                    \end{align*}
                    We set $x = u^{w_j}$ and $y = v^{w_k}$, the left side becomes
                    \begin{align*}
                        \tau_{jk}(\textbf{w})&=\frac{1}{w_jw_k} \int_{[0,1]^2} C(\textbf{1}_{jk}(x, y))\dot{C}_j(x^{w_1/w_j}, \dots, x^{w_d/w_j}) \dot{C}_k(y^{w_1/w_k}, \dots, y^{w_d/w_k}) x^{(1-w_j)/w_j}y^{(1-w_k)/w_k}dxy
                        \\
                        &= \frac{\mu_j(\textbf{w}) \mu_k(\textbf{w})}{w_jw_k}\int_{[0,1]^2}C(\textbf{1}_{jk}(x, y)) x^{A_j(\textbf{w})+(1-w_j)/w_j-1}y^{A_k(\textbf{w})+(1-w_k)/w_k-1}dxy.
                    \end{align*}
                    Now, we set $x = w^{1-s}$ and $y=w^s$ and we obtain
                    \begin{align*}
                        &\tau_{jk}(\textbf{w}) = \\ &\frac{\mu_j(\textbf{w}) \mu_k(\textbf{w})}{w_jw_k} \int_{[0,1]} \bigg[A(\textbf{0}_{jk}(1-s, s)) + (1-s)\left(A_j(\textbf{w})+\frac{1-w_j}{w_j}-1 \right)+ s\left(A_k(\textbf{w}) + \frac{1-w_k}{w_k}-1 \right)+1\bigg]^{-2}ds.
                    \end{align*}
                    The right side of Equation (\ref{cov_ij}) is given by
                    \begin{equation*}
                       \int_{[0,1]^2}  u^{w_j}v^{w_k} \dot{C}_j(u^{w_1}, \dots, u^{w_d})\dot{C}_k(v^{w_1}, \dots, v^{w_d})duv = \frac{\mu_j(\textbf{w}) \mu_k(\textbf{w})}{(1+A(\textbf{w}))^2}.
                    \end{equation*}
                    Hence the result for $\sigma_{d+1}^2(\textbf{w})$. Using the same techniques, we show that for $j\in \{1,\dots,d\}$
                    \begin{equation*}
                        \sigma_j^2(\textbf{w}) = \int_{[0,1]^2} (u\wedge v)^{w_j} - u^{w_j} v^{w_j}duv = \frac{1}{(1+w_j)^2} \frac{w_j}{2+w_j}.
                    \end{equation*}
                    For $j < k$, we compute
                    \begin{align*}
                        \sigma_{jk}(\textbf{w}) &= \int_{[0,1]^2} C(\boldsymbol{1}_{jk}(u^{w_j}, v^{w_k})) - u^{w_j}v^{w_k}duv \\
                        &= \frac{1}{w_jw_k} \int_{[0,1]} \left[ A(\boldsymbol{0}_{jk}(1-s,s)) + (1-s)\frac{1-w_j}{w_j} +s\frac{1-w_k}{w_k}+1 \right]^{-2}ds - \frac{1}{1+w_j}\frac{1}{1+w_k}.
                    \end{align*}
                    Let $j \in \{1,\dots,d\}$, thus
                    \begin{align*}
                        \sigma_j^{(1)}(\textbf{w}) &= \int_{[0,1]^2} C\left(u^{w_1}, \dots, (u \wedge v)^{w_j}, \dots, u^{w_d}\right) - C(u^{w_1}, \dots, u^{w_d})v^{w_j}ds \\
                        &= \frac{1}{w_j(1-w_j)} \int_{[0,1]} \left[ A(\textbf{z}_j(1-s) + (1-s)\frac{1-w_j}{w_j} + s\frac{w_j}{1-w_j}+1\right]^{-2}ds + \frac{1}{1+A(\textbf{w})} \left[ \frac{1}{2+A(\textbf{w})} - \frac{1}{1+w_j}\right].
                    \end{align*}
                    Now, for $\sigma^{(2)}_{jk}$, we have to consider three cases :
                    \begin{itemize}
                        \item if $j = k$, we directly have
                        \begin{align*}
                           \hspace{-7.4cm} \sigma^{(2)}_{jk}(\textbf{w}) &= 0,
                        \end{align*}
                        \item if $j < k$, we obtain  
                        \begin{align*}
                        &\sigma^{(2)}_{jk}(\textbf{w}) \\ &= \frac{\mu_k(\textbf{w})}{w_j w_k} \int_{[0,1]} \bigg[ A(\textbf{0}_{jk}(1-s,s)) + (1-s)\frac{1-w_j}{w_j} + s\left(A_k(\textbf{w}) + \frac{1-w_k}{w_k} -1\right) + 1\bigg]^{-2}ds  - \frac{\mu_k(\textbf{w})}{1 + A(\textbf{w})} \frac{1}{1+w_j},
                    \end{align*}
                    \item   if $j > k$, we have
                    \begin{align*}
                        &\sigma^{(2)}_{jk}(\textbf{w})  \\ &= \frac{\mu_k(\textbf{w})}{w_j w_k} \int_{[0,1]} \bigg[ A(\textbf{0}_{kj}(1-s,s)) + s\frac{1-w_j}{w_j} + (1-s)\left(A_k(\textbf{w}) + \frac{1-w_k}{w_k} -1\right) + 1\bigg]^{-2}ds  - \frac{\mu_k(\textbf{w})}{1 + A(\textbf{w})} \frac{1}{1+w_j}.
                    \end{align*}
                    \end{itemize}
                    Hence the statement.
                \end{proof}
                \label{proof_strong_consistency}
                The following lines will give some details to establish the explicit formula of the asymptotic variance when we suppose that components of the random vector $\textbf{X}$ are independent. In this framework, we have that $\mu_j(\textbf{w}) = 1$ for every $j \in \{1,\dots,d\}$ and thus $\dot{C}_j(u^{w_1},\dots, u^{w_d}) = u^{1-w_j}$. Furthermore, in the independent case, most of the integrals are reduced to zero.
                \label{proof_coro_independent}
                \begin{proof}[Proof of Corollary \ref{coro_independent}]
                    In the term $\sigma_{d+1}^2$ given in Equation \eqref{Var_N_C}, only the terms $\gamma_1^2$, $\gamma_j^2$ and $\gamma_{1j}$ matter because, in the independent case :
                    \begin{equation*}
                        \tau_{jk}(\textbf{w}) = \int_{[0,1]^2}\left( u^{w_j} v^{w_k} - u^{w_j}v^{w_k}\right)\dot{C}_j(u^{w_1}, \dots, u^{w_d})\dot{C}_k(v^{w_1}, \dots, v^{w_d})duv = 0.
                    \end{equation*}
                    For $\gamma_{1j}$, we have to compute
                    \begin{equation*}
                        \gamma_{1j}(\textbf{w}) = 2\int_{[0,1]\times [0,v]} u(1-v^{w_j})v^{1-w_j} duv
                        = \frac{1}{4} \frac{w_j}{4-w_j}.
                    \end{equation*}
                    For $\gamma_1^2$ and $\gamma_j^2$, we just have to set $A(\textbf{w}) = 1$ in their respective expressions to obtain : 
                    \begin{equation*}
                        \gamma_1^2(\textbf{w}) = \frac{1}{12}, \quad \gamma_{j}^2 = \frac{1}{4} \frac{w_j}{4-w_j}.
                    \end{equation*}
                    We thus have 
                    \begin{equation*}
                        \sigma_{d+1}^2(\textbf{p}, \textbf{w}) = \frac{1}{4} \left(\frac{1}{3p} - \sum_{j=1}^d p_j^{-1} \frac{w_j}{4-w_j} \right).
                    \end{equation*}
                Other computations follow from the same arguments.
                \end{proof}
                 We are now going to prove Proposition \ref{strong_consistency}. The strong consistency of the our estimators will be established in a two-step process : first, we prove the strong consistency of the estimator $\nu_n(\textbf{w})$ which is the nonparametric estimator of the $\textbf{w}$-madogram with known margins and, second, we show that the limit of
                \begin{equation*}
                    \underset{j \in \{1, \dots,d\}}{\sup}\underset{i \in \{1, \dots,n\}}{\sup} \left| \widetilde{U}_{i,j}^{1/w_j} - \left\{F_j(\tilde{X}_{i,j})\right\}^{1/w_j} \right|,
                \end{equation*}
                is zero almost surely. Before going into the main arguments, we need the following lemma.
                    \begin{lemma}
                        \label{lemma_2}
                            We have, $\forall i \in \{1,\dots,n\}$
                            \begin{equation*}
                                \left| \bigvee_{j=1}^d \widetilde{U}_{i,j}^{1/w_j} - \bigvee_{j=1}^d \big\{F_j(X_j)\big\}^{1/w_j}\right| \leq \underset{j\in \{1,\dots,d\}}{\sup} \left|\widetilde{U}_{i,j}^{1/w_j}  - \big\{ F_j(X_j)\big\}^{1/w_j} \right|.
                            \end{equation*}
                        \end{lemma}
                    The proof of Lemma \ref{lemma_2} is postponed to \ref{proof_lemmeta}.
                \begin{proof}[Proof of Proposition \ref{strong_consistency}] 
                We prove it for $\hat{\nu}_n^{\mathcal{H}}(\textbf{w})$ as the strong consistency for $\hat{\nu}_n^{\mathcal{H}*}(\textbf{w})$ uses  the same arguments. The estimator $\hat{\nu}_n^\mathcal{H}(\textbf{w})$ in \eqref{hybrid_lambda_fmado} is strongly consistent since it holds   
                    \begin{align*}
                        \left| \hat{\nu}_n^\mathcal{H}(\textbf{w}) - \nu(\textbf{w})\right| &= \left| \hat{\nu}^\mathcal{H}_n(\textbf{w}) - \nu_n(\textbf{w}) + \nu_n(\textbf{w}) - \nu(\textbf{w}) \right|         \leq \left| \hat{\nu}^\mathcal{H}_n(\textbf{w}) - \nu_n(\textbf{w}) \right| + \left| \nu_n(\textbf{w}) - \nu(\textbf{w})\right|,
                    \end{align*}
                    where 
                    \begin{equation*}
                        \nu_n(\textbf{w}) = \frac{1}{N} \sum_{i=1}^n \left[ \left( \bigvee_{j=1}^d \left\{ F_j(\tilde{X}_{i,j}) \right\}^{1/w_j} - \frac{1}{d} \sum_{j=1}^d \left\{ F_j(\tilde{X}_{i,j}) \right\}^{1/w_j} \right)n_i \right].
                    \end{equation*}
                    By direct application of Assumption \ref{Cond_2} and the law of large number, we have that
                    \begin{equation*}
                        \left| \nu_n(\textbf{w}) - \nu(\textbf{w})\right| \overunderset{a.s.}{n \rightarrow \infty}{\rightarrow} 0
                    \end{equation*}
                    For the second term, we write :
                    \begin{align*}
                        \left| \hat{\nu}_n^{\mathcal{H}}(\textbf{w}) - \nu(\textbf{w}) \right| \leq& \frac{1}{N} \sum_{i=1}^n \left| \bigvee_{j=1}^d\left\{\hat{F}_{n,j}(\tilde{X}_{i,j})\right\}^{1/w_j} - \bigvee_{j=1}^d \left\{F_j(X_j)\right\}^{1/w_j} \right|n_i \\ &+ \frac{1}{Nd} \sum_{i=1}^n \sum_{j=1}^d \left| \left\{\hat{F}_{n,j}(\tilde{X}_{i,j})\right\}^{1/w_j} - \left\{F_j(\tilde{X}_{i,j})\right\}^{1/w_j}\right|n_i \\
                        \leq& 2\underset{j \in \{1, \dots,d\}}{\sup}\underset{i \in \{1, \dots,n\}}{\sup} \left| \left\{\hat{F}_{n,j}(\tilde{X}_{i,j})\right\}^{1/w_j} - \left\{F_j(\tilde{X}_{i,j})\right\}^{1/w_j} \right|,
                    \end{align*}
                    where we used Lemma \ref{lemma_2} to obtain the second inequality. The right term converges almost surely to zero by Glivencko-Cantelli Theorem and the uniform continuity of $x \mapsto x^{1/w_j}$ on $[0,1]$.
                \end{proof}
                Finally, we give some elements to establish Corollary \ref{asymptotic_distribution_pickands}. The strong consistency follows directly from the stability of the almost surely convergence through a continuous fuction. The weak convergence comes down from the functional Delta method (see, \textit{e.g.}, Theorem 3.9.4 of \cite{vaartwellner96book}) and from result   in Proposition \ref{Boulin}.
                \label{proof_asymptotic_distribution_pickands}
                \begin{proof}[Proof of Corollary \ref{asymptotic_distribution_pickands}]
                    Applying the functional Delta method, we have as $n \rightarrow \infty$,
                    \begin{align*}
                        \sqrt{n} \left(\hat{A}_n^{\mathcal{H}*}(\textbf{w}) - A(\textbf{w})\right) &\rightsquigarrow -\left(1+A(\textbf{w})\right)^2 \bigg\{\frac{1}{d}\sum_{j=1}^d \left( 1 + \lambda_j(\textbf{w})(d-1) \right)\int_{[0,1]} \alpha(\textbf{1}_j(x^{w_j})) - \beta_j(x^{w_j})dx \\ &- \int_{[0,1]} S_C(x^{w_1}, \dots, x^{w_d}dx)\bigg\}_{\emph{\textbf{w}} \in \Delta^{d-1}}.
                    \end{align*}
                    For a fixed $\textbf{w} \in \Delta^{d-1}$, as a linear transformation of a tight Gaussian process, it follows that
                    \begin{align*}
                        \sqrt{n} \left(\hat{A}_n^{\mathcal{H}*}(\textbf{w}) - A(\textbf{w})\right) \overunderset{d}{n \rightarrow \infty}{\rightarrow} \mathcal{N}\left(0, \mathcal{V}(\textbf{p}, \textbf{w})\right),
                    \end{align*}
                    with, $\mathcal{V}(\textbf{p}, \textbf{w})$ equals by definition
                    \begin{align*}
                        &Var\left(-\left(1+A(\textbf{w})\right)^2 \bigg\{\frac{1}{d}\sum_{j=1}^d \left( 1 + \lambda_j(\textbf{w})(d-1) \right)\int_{[0,1]} \alpha(\textbf{1}_j(x^{w_j})) - \beta_j(x^{w_j})dx 
                    - \int_{[0,1]} S_C(x^{w_1}, \dots, x^{w_d})\bigg\} dx \right) \\
                        &= (1+A(\textbf{w}))^4 \mathcal{S}^{\mathcal{H}*}(\textbf{p}, \textbf{w}),
                    \end{align*}
                    where we used Proposition \ref{Boulin} to conclude.
                \end{proof}
            \subsection{Proofs of auxiliary results}
                \label{proof_lemmeta}
                \label{proof_covariance_function_hybrid}
                \begin{proof}[Proof of Lemma \ref{lemma_1}]
                    Following \cite{segers2014hybrid} Example 3.5, we consider the functions from $\{0,1\}^d \times \mathbb{R}^d$ into $\mathbb{R}$ : for $\textbf{x} \in \mathbb{R}^d$, and $j \in \{1,\dots,d\}$
                    \begin{align*}
                        &f_j(\textbf{I}, \textbf{X}) = \mathds{1}_{\{I_j = 1\}}, \quad g_{j,x_j}(\textbf{I}, \textbf{X}) = \mathds{1}_{\{X_j \leq x_j, I_j = 1\}}, 
                        \quad f_{d+1} = \Pi_{j=1}^d f_j, \quad g_{d+1,\textbf{x}} = \Pi_{j=1}^d g_{j,x_j}.&
                    \end{align*}
                    Let $P$ denote the common distribution of the tuple $(\textbf{I}, \textbf{X})$. The collection of functions
                    \begin{equation*}
                        \mathcal{F} = \{f_1, \dots, f_{d}, f_{d+1}\} \cup \bigcup_{j=1}^d \{g_{j,x_j}, x_j \in \mathbb{R}\} \cup \{g_{d+1, \textbf{x}}, \textbf{x} \in \mathbb{R}^d\}
                    \end{equation*}
                    is a finite union of VC-classes and thus $P$-Donsker (see Chapter 2.6 of \cite{vaartwellner96book}). The empirical process $\mathbb{G}_n$ defined by
                    \begin{equation*}
                        \mathbb{G}_n(f) = \sqrt{n}\left(\frac{1}{n} \sum_{i=1}^n f(\textbf{I}_i, \textbf{X}_i) - \mathbb{E}[f(\textbf{I}_i, \textbf{X}_i)]\right), \quad f \in \mathcal{F},
                    \end{equation*}
                    converges in $\ell^\infty(\mathcal{F})$ to a $P$-brownian bridge $\mathbb{G}$. For $\textbf{x} \in \mathbb{R}^d$,
                    \begin{align*}
                        \hat{F}_{n,j}(x_j) = \frac{p_j F_j(x_j) + n^{-1/2}\mathbb{G}_{n}g_{j,x_j}}{p_j + n^{-1/2}\mathbb{G}_nf_j}, \\
                        \hat{F}_n(\textbf{x}) = \frac{p F(\textbf{x}) + n^{-1/2}\mathbb{G}_ng_{d+1,\textbf{x}}}{p + n^{-1/2}\mathbb{G}_n f_{d+1}}.
                    \end{align*}
                    We obtain for the second one
                    \begin{align*}
                        p\left( \hat{F}_n(\textbf{x}) - F(\textbf{x}) \right) &= n^{-1/2} \left(\mathbb{G}_n(g_{d+1,\textbf{x}}) - \hat{F}_n(\textbf{x})\mathbb{G}_n(f_{d+1})\right) \\
                        &= n^{-1/2}\left(\mathbb{G}_n(g_{d+1,\textbf{x}}-F(\textbf{x})f_{d+1})\right) - n^{-1/2}\mathbb{G}_n(f_{d+1})(\hat{F}_n(\textbf{x}) - F(\textbf{x})).
                    \end{align*}
                    We thus have
                    \begin{equation*}
                        \sqrt{n} \left(\hat{F}_n(\textbf{x}) - F(\textbf{x}) \right) = p^{-1}\left(\mathbb{G}_n(g_{d+1,\textbf{x}}-F(\textbf{x})f_{d+1})\right) - p^{-1}\mathbb{G}_n(f_{d+1})(\hat{F}_n(\textbf{x}) - F(\textbf{x})).
                    \end{equation*}
                    Applying the central limit theorem and Assumption \ref{Cond_2} gives that $\mathbb{G}_n(f_{d+1}) \overset{d}{\rightarrow} \mathcal{N}(0, \mathbb{P}(f_{d+1} - \mathbb{P}f_{d+1})^2)$, the law of large numbers gives also $\hat{F}_n(\textbf{x}) - F(\textbf{x}) = \circ_{\mathbb{P}}(1)$. Using Slutsky's lemma gives us
                    \begin{equation*}
                        \sqrt{n} \left(\hat{F}_n(\textbf{x}) - F(\textbf{x}) \right) = p^{-1}\left(\mathbb{G}_n(g_{d+1,\textbf{x}}-F(\textbf{x})f_{d+1})\right) + \circ_{\mathbb{P}}(1).
                    \end{equation*}
                    Similar reasoning might be applied to the margins, as a consequence, Assumption  \ref{Cond_3} is fulfilled with for $\textbf{u} \in [0,1]^d$,
                    \begin{align*}
                        &\beta_j(u_j) = p_j^{-1} \mathbb{G}\left(g_{j, F_j^\leftarrow(u_j)} - u_j f_j \right), \\
                        &\alpha(\textbf{u}) = p^{-1} \mathbb{G} \left(g_{d+1, \textbf{F}_d^\leftarrow(\textbf{u})} - C(\textbf{u})f_{d+1} \right).
                    \end{align*}
                    Let us compute one covariance function, the method still the same for the others, without loss of generality, suppose that $j<k$, we have for $u_j, v_k \in [0,1]$
                    \begin{align*}
                        cov(\beta_j(u_j), \beta_k(v_k)) &= \mathbb{E}\left[p_j^{-1} \mathbb{G}\left(g_{j, F_j^\leftarrow(u_j)} - u_j f_j \right) p_k^{-1} \mathbb{G}\left(g_{k, F_k^\leftarrow(v_k)} - v_k f_k \right)\right] \\
                        &= \frac{1}{p_j p_k}\mathbb{E}\left[\mathbb{G}\left(g_{j, F_j^\leftarrow(u_i)} - u_j f_j \right) \mathbb{G}\left(g_{k, F_k^\leftarrow(v_j)} - v_k f_k \right)\right] \\
                        &= \frac{1}{p_jp_k} \mathbb{P}\left\{ X_j \leq F_j^\leftarrow(u_j), X_k \leq F_k^\leftarrow(v_k), I_j =1, I_k= 1 \right\} - \frac{p_{jk}}{p_jp_k} u_jv_k \\
                        &= \frac{1}{p_jp_k}\mathbb{P}\left\{ X_j \leq F_j^\leftarrow(u_j), X_k \leq F_k^\leftarrow(v_k)\right\} \mathbb{P}\left\{ I_j = 1, I_k = 1\right\}-\frac{p_{jk}}{p_jp_k} u_jv_k \\
                        &= \frac{p_{jk}}{p_jp_k} \left( C(\textbf{1}_{jk}(u_j,v_k)) -u_jv_k\right).
                    \end{align*}
                    Hence the result.
                \end{proof}
                \begin{proof}[Proof of Lemma \ref{lemma_2}]
                The lemma becomes trivial once we write, $\forall i \in \{1,\dots,n\}$ and $j \in \{1,\dots,d\}$
                \begin{align*}
                    \widetilde{U}_{i,j}^{1/w_j} &= \left\{F_j(X_j)\right\}^{1/w_j} + \widetilde{U}_{i,j}^{1/w_j} - \left\{F_j(X_j)\right\}^{1/w_j} \\
                    &\leq \left\{F_j(X_j)\right\}^{1/w_j} + \underset{j\in \{1,\dots,d\}}{\sup} \left|\widetilde{U}_{i,j}^{1/w_j} - \left\{F_j(X_j)\right\}^{1/w_j}\right| \\
                    &\leq \bigvee_{j=1}^d \left\{F_j(X_j)\right\}^{1/w_j} + \underset{j\in \{1,\dots,d\}}{\sup} \left|\widetilde{U}_{i,j}^{1/w_j} - \left\{F_j(X_j)\right\}^{1/w_j}\right|.
                \end{align*}
                Taking the max over $j \in \{1,\dots,d\}$ gives
                \begin{align*}
                    \bigvee_{j=1}^d\widetilde{U}_{i,j}^{1/w_j} - \bigvee_{j=1}^d \left\{ F_j(X_j)\right\}^{1/w_j} \leq \underset{j\in \{1,\dots,d\}}{\sup} \left|\widetilde{U}_{i,j}^{1/w_j} - \left\{F_j(X_j)\right\}^{1/w_j}\right|.
                \end{align*}
                Moreover, by symmetry of $\widetilde{U}_{i,j}$ and $F_j$, the second one follows similarly.
            \end{proof}
            \bibliographystyle{apalike}
            \bibliography{trial}
\end{document}